\RequirePackage{fix-cm}
\documentclass[twocolumn]{svjour3}          
\smartqed  
\usepackage{graphicx}
\usepackage{mathptmx}      
\usepackage{lineno,hyperref}
\usepackage{booktabs, multicol, multirow} 
\usepackage{amsmath} 
\usepackage{amssymb} 
\usepackage{pgfplots}
\usepackage{subcaption}
\usepackage{mathtools}
\usepackage{algorithm} 
\usepackage{algpseudocode}
\usepackage{amsfonts}
\usepackage[numbers]{natbib}

\usepackage{hyphenat}

\journalname{Computational Mechanics}
\graphicspath{{./}{./figs/}}

  \def\clap#1{\hbox to 0pt{\hss#1\hss}}

\usepackage{bm}
\providecommand{\mat}[1]{\bm{#1}}%
\renewcommand{\vec}[1]{\mathbf{#1}}

\providecommand{\mC}{\ensuremath{\mat{C}}}

\providecommand{\mF}{\ensuremath{\mat{F}}}

\providecommand{\mH}{\ensuremath{\mat{H}}}
\providecommand{\mI}{\ensuremath{\mat{I}}}

\providecommand{\mL}{\ensuremath{\mat{L}}}

\providecommand{\mP}{\ensuremath{\mat{P}}}
\providecommand{\mQ}{\ensuremath{\mat{Q}}}
\providecommand{\mR}{\ensuremath{\mat{R}}}

\providecommand{\mU}{\ensuremath{\mat{U}}}
\providecommand{\mV}{\ensuremath{\mat{V}}}
\providecommand{\mW}{\ensuremath{\mat{W}}}

\providecommand{\vc}{\ensuremath{\vec{c}}}

\providecommand{\vu}{\ensuremath{\vec{u}}}

\pgfplotsset{compat=1.14}

\begin{document}

\title{Bi-fidelity Reduced Polynomial Chaos Expansion for Uncertainty Quantification
}



\author{Felix Newberry         \and
        Jerrad Hampton         \and
        Kenneth Jansen         \and
        Alireza Doostan     
}


\institute{Felix Newberry \and Kenneth Jansen \and Alireza Doostan \at
              Ann and H.J. Smead Aerospace Engineering Sciences Department, University of Colorado, Boulder, CO 80303, USA, \\\email{Alireza.Doostan@colorado.edu}
           \and
           Jerrad Hampton 
                \at Centre Internacional de M\`etodes Num\`erics a l'Enginyeria, Esteve Terrades 5, 08860 Castelldefels, Spain
}

\date{Received: mm/dd/yyyy / Accepted: mm/dd/yyyy}

\maketitle

\begin{abstract}
\label{sec:abstract}
%

A ubiquitous challenge in design space explora\hyp{}tion or uncertainty quantification of complex engineering problems is the minimization of computational cost. A useful tool to ease the burden of solving such systems is model reduction. This work considers a stochastic model reduction method  (SMR), in the context of polynomial chaos (PC) expansions, where low-fidelity (LF) samples are leveraged to form a stochastic reduced basis. The reduced basis enables the construction of a bi-fidelity (BF) estimate of a quantity of interest from a small number of high-fidelity (HF) samples. A successful BF estimate approximates the quantity of interest with accuracy comparable to the HF model and computational expense close to the LF model. We develop new error bounds  for the SMR approach and present a procedure to practically utilize these bounds in order to assess the appropriateness of a given pair of LF and HF models for BF estimation. The effectiveness of the SMR approach, and the utility of the error bound are presented in three numerical examples. 


\end{abstract}

\section{Introduction}
\label{sec:introduction_pc}

A core motivation in engineering design is to understand the behavior of some quantity of interest (QoI) as a function of uncertain inputs. Stochastic variables may arise from uncertainties in measurements, model parameters{\color{black},} or boundary and initial conditions. Limited understanding of the influence that stochastic inputs have on the QoI may yield a surplus of confidence or mistaken reluctance to trust the model predictions. The field of uncertainty quantification (UQ) addresses this problem and has been the subject of much research \cite{ghanem1991stochastic,le2010spectral,xiu2010numerical}. 

A useful technique in UQ problems is to approximate the QoI with an expansion in multivariate orthogonal polynomials, known as the polynomial chaos (PC) expansion \cite{ghanem1991stochastic, xiu2002wiener}. In this work, we assume a $d$-dimensional vector of random inputs $\boldsymbol{\Xi} \coloneqq (\Xi_1, \dots, \Xi_d )$ with joint probability density function $f(\boldsymbol{\xi})$ and set of possible realizations $\Omega$. We consider the vector valued QoI, $\vu(\boldsymbol{\Xi}) \in \mathbb{R}^M$, assumed to have finite variance and defined over a spatial domain of the problem. We note that while this work focuses on spatial QoIs, the method described extends to temporal or spatio-temporal QoIs. The PC expansion approximates the vector QoI as
\begin{equation}
\vu(\boldsymbol{\Xi}) = \sum_{j=1}^\infty \vc_{j} \psi_{j}(\boldsymbol{\Xi}), \\
\label{eq:pce_inf}
\end{equation} 
where $\psi_{j}(\boldsymbol{\Xi})$ is a multivariate orthogonal polynomial evaluated at the random inputs and weighted by deterministic coefficients $\vc_{j} \in \mathbb{R}^M$. The polynomials $\psi_{j}(\boldsymbol{\Xi})$ are chosen to be orthogonal with respect to the probability measure $f(\boldsymbol{\xi})$. For instance, if $\boldsymbol{\Xi}$ follows a jointly uniform or Gaussian distribution, then $\psi_{j}(\boldsymbol{\Xi})$ are multivariate Legendre or Hermite polynomials, respectively \cite{xiu2002wiener}. We assume $\psi_{j}(\boldsymbol{\Xi})$ are normalized such that $\mathbb{E} [\psi^2_{j}(\boldsymbol{\Xi})] =1 $, where $\mathbb{E} [\cdot ]$ represents the mathematical expectation operator. The expansion (\ref{eq:pce_inf}) is truncated to
\begin{equation}
\vu(\boldsymbol{\Xi}) = \sum_{j=1}^P \vc_{j} \psi_{j}(\boldsymbol{\Xi}) + \boldsymbol{\delta}_{P}(\boldsymbol{\Xi}) \approx \sum_{j=1}^P \vc_{j} \psi_{j}(\boldsymbol{\Xi}), \\
\label{eq:pce_trunc}
\end{equation} 
where $\boldsymbol{\delta}_{P}$ denotes the PC truncation error. The truncated expansion (\ref{eq:pce_trunc}) is accurate provided the coefficients $\vc_{j}$ decay to zero in a properly ordered basis and that the QoI depends smoothly on the inputs $\boldsymbol{\Xi}$. An expansion with total order $p$ and stochastic dimension $d$ has $P = (p+d)!/(p!d!)$ basis functions. As $P \rightarrow \infty$, for a sufficiently smooth $\vu(\boldsymbol{\Xi})$, the PC expansion converges in the mean-square sense to $\vu$. 

The coefficients $\vc_{j}$ are a valuable tool for approximating statistics, constructing surrogate models that may integrate or differentiate our QoI or performing sensitivity analysis. 
In this work, we assemble these coefficients into the matrix $\mC :=[\vc_{1},\dots,\vc_{P}] \in \mathbb{R}^{M \times P}$. 
A common approach to determine $\mC$ is to construct a regression problem with Monte Carlo samples of the QoI. 
We denote individual realizations of $\boldsymbol{\Xi}$ as $\boldsymbol{\xi}_{i}$, and consider a set of $N$ input samples as $\{ \boldsymbol{\xi}_{i}\}^N_{i=1}$ and correspondingly QoI samples $\{ \vu (\boldsymbol{\xi}_{i}) \}^N_{i=1}$ organized in the data matrix $\mU \coloneqq [ \vu(\boldsymbol{\xi}_1), \dots, \vu(\boldsymbol{\xi}_N)]\in\mathbb{R}^{M\times N}$. 
We seek to solve for $\mC$ in the linear system 
\begin{equation}
\mC \boldsymbol{\Psi}\approx \mU,
\end{equation}
where $\boldsymbol{\Psi}(j,i) \coloneqq \psi_{j}(\boldsymbol{\xi}_{i} )$. If the resulting regression problem is over-determined, with $N>P$,  then we employ least squares approximation whereas if it is under-determined, with $N<P$, then we apply compressed sensing \cite{doostan2011non, peng2014weighted, hampton2015compressive, diaz2018sparse}. The analysis of this work is based on the least squares optimization problem, 
\begin{equation}
\underset{\mC}{\min} \;\; \Vert  \mC\boldsymbol{\Psi}  - \mU \Vert_F,
\label{eq:LS}
\end{equation}
whose solution $\mC$ may be computed from the normal equation $  \mC\boldsymbol{\Psi} \boldsymbol{\Psi}^T =  \mU\boldsymbol{\Psi}^T $. In (\ref{eq:LS}), $\Vert \cdot \Vert_F$ denotes the Frobenius norm. This technique essentially assumes a fixed basis which means modeling a QoI in high dimensions $d$ and/or with high polynomial order $p$ leads to a correspondingly high number of PC basis functions $P$ and to potentially expensive simulations that require many HF samples.

Sparse PC expansions, e.g., via compressed sensing, tackle this issue through exploiting sparsity in the PC coefficients $\vc_{j}$ \cite{donoho2006compressed, candes2008introduction, doostan2011non, blatman2010adaptive, blatman2011adaptive, mathelin2012compressed, yan2012stochastic, yang2013reweighted, peng2014weighted, hampton2015compressive}, and require relatively smaller sample sizes. {\color{black}For instance, solving the $\ell_{1,2}$-minimization problem 
\begin{equation}
\underset{\mC}{\min} \;\; \Vert \mC \Vert_{1,2} \quad \text{subject to} \quad \Vert \mC\boldsymbol{\Psi}   - \mU \Vert_F \leq \kappa,
    \label{eq:ell_1}
\end{equation}
where $\Vert\mC \Vert_{1,2}:=\left(\sum_{i}\Vert\bm{C}(i,:)\Vert_1^2\right)^{1/2}$, $\bm{C}(i,:)$ denotes} the $i$th row of $\bm C$, $\Vert\cdot\Vert_1$ is the $\ell_1$-norm of a vector, and $\kappa$ is a tolerance of solution inaccuracy \cite{donoho2006compressed, candes2008introduction, hampton2015compressive}.
In the present work, we seek to further reduce the sampling requirement through an SMR framework that leverages LF model evaluations. The central theme of this work is two fold. Firstly, if one knew a stochastic reduced basis $\{\eta_j(\bm\Xi)\}_{j=1}^r$, with $r\ll P$, associated with the subspace on which the QoI lives then a regression problem of the form (\ref{eq:LS}) would require a relatively smaller number $N$ of QoI realizations, thus leading to a reduced computational expense. Secondly, such a reduced basis may be identified -- in an approximate sense -- from LF models of the problem.

In practice, for a given physical system, models of differing fidelity are available. HF models, that accurately describe the underlying physics, may require significant computational expense that becomes infeasible when many evaluations are necessary. In contrast, LF models are likely to provide a less accurate prediction of the problem physics at a comparatively affordable computational cost. For instance, in simulations of transient fluid flows that feature complex geometries, HF models must have fine spatial and temporal discretizations to resolve sharp gradient regions, such as boundary layers or separated flow. A LF counterpart with comparatively coarse discretizations will provide less accurate predictions that are generated more quickly. 
Multi-fidelity techniques exploit this inherent variety in model fidelity to reduce the computational expense involved in engineering design \cite{kennedy2000predicting, fernandez2016review, peherstorfer2018survey}. Multi-fidelity methods have seen statistical applications such as co-kriging \cite{forrester2007multi, laurenceau2008building} and numerous recursive approaches designed to improve accuracy and mitigate computational complexity \cite{kleiber2013parameter, le2014recursive, le2015cokriging,perdikaris2015multi, perdikaris2016multifidelity,parussini2017multi}, among others. In recent years  progress in multi-fidelity UQ has been made in areas such as multi-level Monte Carlo \cite{giles2013multilevel,giles2008multilevel,cliffe2011multilevel} and PC expansions \cite{eldred2009recent, ng2012multifidelity, palar2015decomposition, padron2016multi}, the subject of this paper. These approaches typically employ a large number of LF samples and a comparatively small number of HF samples to perform an additive and/or multiplicative correction to the LF model. 

\subsection{Contributions of this work} 
\label{sec:contributions_pc}

Previous SMR work that inspires the present study constructs a small polynomial representation of the HF solution though determining a reduced basis with the Karhunen Lo\'eve (KL) expansion of the LF solution \cite{doostan2007stochastic,ghanem2007efficient}. This method then estimates the HF solution via Galerkin projection of the governing equations onto the span of the reduced polynomial basis. A non-intrusive approach that determines PC coefficients via regression as opposed to Galerkin or Petrov-Galerkin projection is  proposed in \cite{raisee2015non}. The key assumption to these methods is that the QoI admits a low-rank covariance, hence the existence of a reduced basis, and that the decay of the covariance eigen-values in the LF and HF models are similar.

At the core of our contribution is the error analysis of approximations using the BF reduced basis. The analysis leverages bounds for PC sampling methods and a related BF approach in which a reduced basis and corresponding interpolation rule are identified from LF data \cite{hampton2015coherence, hampton2018practical}. The derived error estimates apply to the BF approaches of \cite{doostan2007stochastic, ghanem2007efficient, raisee2015non} and provide practitioners with a means to assess the quality of a LF model in leading to accurate BF estimates. One of the error estimates can be generated using a small number of HF samples and empirically leads to sharp estimates of the true BF error.

In Sections \ref{sec:stochastic_model_reduction} and \ref{sec:KL}, we introduce the key components that make up {\color{black}SMR}. Next, in Section \ref{sec:error_bound_analysis}, we derive a theoretical error bound that provides insight into the features that influence the error, and, in \ref{sec:pract_enhance}, practical enhancement for user implementation. Sections \ref{sec:case_1_ldc}, \ref{sec:case_2_cylinder}{\color{black},} and \ref{sec:case_3_airfoil} demonstrate the effectiveness of the {\color{black}SMR} method and utility of the error bound in three numerical examples, namely a lid-driven cavity flow, heated flow past a cylinder and flow past a NACA 4412 airfoil. Finally, Section \ref{sec:conclusion_pc} gives a brief summary of this study's conclusions. 

\section{Method detail} 
\label{sec:method_detail_pc}

We assume vector valued QoIs $\vu$ that are defined over a spatial domain of the problem and exhibit a low-dimensional subspace on which the solution to (\ref{eq:LS}) lies. While the QoIs may represent an entire domain, such as an airfoil immersed in a velocity field, we are typically interested in more refined QoI selection, for instance the coefficient of pressure along the airfoil surface. 
We denote the LF and the corresponding HF QoIs  as  $\vu^L \coloneqq \{ u^L_1, \dots, u^L_m \} \in \mathbb{R}^m$ and $\vu^H \coloneqq \{ u^H_1, \dots, u^H_M \} \in \mathbb{R}^M$, respectively, and use this convention to indicate the associated fidelity  of variables throughout this manuscript. Although these vectors may be of different lengths, they are functions of the same stochastic inputs $\boldsymbol{\Xi}$.

We assemble $N$ LF and HF realizations of the QoIs into matrices $\mL \in \mathbb{R}^{m \times N}$ and $\mH \in \mathbb{R}^{M \times N} $, which we refer to as LF and HF data, respectively, such that 
\begin{equation}
    \mL \coloneqq \left[ \begin{array}{cccc} \vu_{1}^L & \vu_{2}^L & \dots & \vu_{N}^L \end{array} \right], \quad \mH \coloneqq \left[ \begin{array}{cccc} \vu_{1}^H & \vu_{2}^H & \dots & \vu_{N}^H \end{array} \right]. 
    \nonumber
\end{equation}
It is noted that $\mH$ and $\mL$ have the same number of columns but may differ in their number of rows as LF and HF models frequently have different spatial resolution. 

The method detail is comprised of two main parts. We first describe the SMR approach on which this study is based. Second, we derive the error bound for this method.

\subsection{Stochastic model reduction (SMR)}
\label{sec:stochastic_model_reduction}
Following \cite{doostan2007stochastic, ghanem2007efficient, raisee2015non}, we seek to establish a low-dimensional subspace on which the solution (\ref{eq:LS}) of the QoI lives. If we consider the full $P$ term polynomial basis, $\{\psi_{j}(\boldsymbol{\Xi})\}_{j=1}^P$ from the PC expansion in (\ref{eq:pce_trunc}), as representing the entire solution space, then our objective is to determine a reduced basis $\{\eta_{i}(\boldsymbol{\Xi})\}_{i=1}^r$ comprised of $r \ll P$ terms as illustrated in Figure \ref{fig:Reduced_basis}. 
{\color{black}This reduced basis can be identified via HF samples, but the evaluation of numerous HF simulations is prohibitively expensive. Instead, we use LF samples to identify the reduced basis and derive error estimates that provide insight into conditions a pair of LF and HF models must satisfy to lead to accurate approximation in the identified reduced basis.}

\begin{figure}[ht!]
\centering
\includegraphics[width=0.4\textwidth]{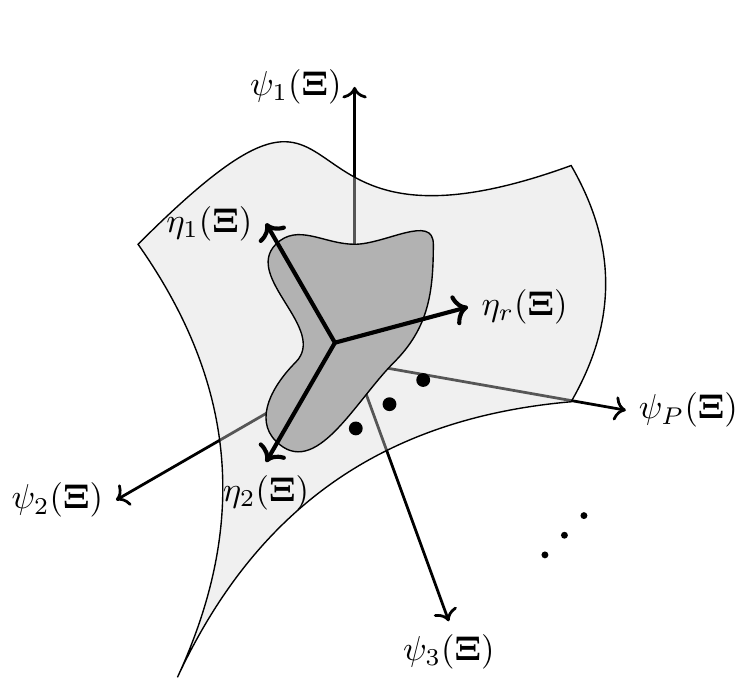}
\caption{The full basis $\{ \psi_j (\boldsymbol{\Xi} )\}^P_{j=1}$ consists of $P$ terms. We seek a reduced basis $\{ \eta_i (\boldsymbol{\Xi} )\}^r_{i=1}$ with $r \ll P$ terms that describes a low-dimensional manifold of the complete solution space. Performing a regression in this basis requires a relatively smaller set of HF samples.}
\label{fig:Reduced_basis}	
\end{figure}

The process to achieve a BF estimate has three key steps. First, we perform a PC expansion of the LF QoI. Second, a reduced basis is obtained with a KL expansion of the LF QoI, that is equated to the LF PC representation. Finally, the reduced basis is utilized with a limited number of HF samples and the BF coefficients are found via least squares regression. We may employ the BF coefficients and basis to construct an estimate of many QoI samples, or retrieve useful statistics directly from the coefficients. We next provide the details of the aforementioned steps of the BF framework.

\subsubsection{PC expansion of LF QoI}
\label{sec:PC_low_data}

We construct a PC expansion of the LF QoI as 
\begin{equation}
\vu^L (\bm{\Xi}) \approx  \sum\limits_{j=1}^{P} {\vc}_{j}^L  \psi_{j} (\bm\Xi), 
\label{eq:pce_low}
\end{equation}
where $\psi_{j} (\bm\Xi)$ are the polynomial basis functions and ${\vc}_{j}^L$ are the corresponding estimated LF coefficients. We solve for the LF coefficients with compressed sensing via $\ell_{1,2}$-minimization following equation (\ref{eq:ell_1}). 
\subsubsection{KL expansion of LF QoI}
\label{sec:KL}
We perform a KL expansion of the LF QoI to identify the dominant stochastic subspace associated with this QoI. The KL expansion is a spectral decomposition of a stochastic process that represents a QoI in terms of the eigenvalues and eigenvectors of its covariance function. We express the (discrete) KL expansion of $\vu^L$ as
 \begin{equation}
\vu^L (\boldsymbol{\Xi} ) \approx \mathbb{E}[\vu^L] + \sum_{i=1}^{r-1} \sqrt{\lambda_{i}^L} \boldsymbol{\varphi}_{i}^L \eta_{i}^L (\boldsymbol{\Xi} ), 
\label{eq:KL_trunc}
\end{equation}
where $\lambda^L_{i}$ and $\boldsymbol{\varphi}^L_{i}$ are eigenvalues and eigenvectors of the covariance matrix of $\vu^L$, approximated by $\sum_{j=2}^P {\vc}_{j}^L ({\vc}_{j}^L)^T$, and $\eta^L_{i}(\bm{\Xi})$ are zero mean orthogonal random variable with unit variance.  
Critical to this method is the assumption that the covariance matrix of the LF QoI has fast decaying eigenvalues, i.e., is low-rank, and that the decay rate is similar to its HF counterpart. Here, $r$ can be found as the minimum integer such that $\sum_{i=1}^{r-1}\lambda_{i}^L/\sum_i \lambda_{i}^L$ is sufficiently close to one. Equating the KL expansion of (\ref{eq:KL_trunc}) and the PC expansions of (\ref{eq:pce_low}) yields %
\begin{equation}
\sum\limits_{j=1}^{P} {\vc}_{j}^L  \psi_{j} (\bm\Xi) \approx \mathbb{E}[ \vu^L ] + \sum_{i=1}^{r-1} \sqrt{\lambda_{i}^L} \boldsymbol{\varphi}_{i}^L \eta_{i}^L (\boldsymbol{\Xi}), 
\nonumber
\end{equation}
where $\mathbb{E}[ \vu^L ] = \vc^L_1$ can be calculated readily from the LF coefficients. 
We solve for the reduced basis $\{\eta_{i}^L(\bm\Xi) \}_{i=1}^{r-1}$ using
\begin{equation}
\eta_{i}^L(\boldsymbol{\Xi} ) \approx \sum_{j=2}^{P} \omega_{ij} \psi_{j}(\boldsymbol{\Xi}),
\label{eq:reduced_basis}
\end{equation}
where $\omega_{ij} = \frac{\langle \boldsymbol{\varphi}_{i}^L , \vc_{j}^L \rangle}{\sqrt{\lambda_{i}^L}}$, with $ \langle \cdot, \cdot \rangle$ denoting the inner product. To account for the mean value of our QoI we prepend our reduced basis with $1$ so that we arrive at a reduced basis $\{ \eta_j^L (\boldsymbol{\Xi} )\}^r_{j=1}:=\{1, \eta_i^L (\boldsymbol{\Xi} )\}^{r-1}_{i=1}$ composed of $r \ll P$ basis functions. In the remainder of this work we omit the superscript $L$ from the reduced basis function for ease of notation, proceeding with $\eta_{j}(\boldsymbol{\Xi} ) := \eta_{j}^L(\boldsymbol{\Xi} )$.

\subsubsection{BF approximation via LF reduced basis lifting}
\label{sec:Bi_fidelity_approximation}

Next, we lift the LF reduced basis $\{ \eta_j (\boldsymbol{\Xi} )\}^r_{j=1}$ and approximate the HF QoI as
\begin{align}
\label{eq:bi_fid}
\vu^H(\boldsymbol{\Xi}) & = \sum_{j=1}^r \vc_{j}^B \eta_{j} (\boldsymbol{\Xi}) + \bm{\delta}_{r}(\bm\Xi),\\ 
& \approx \sum_{j=1}^r \vc_{j}^B \eta_{j} (\boldsymbol{\Xi}): = \vu^{B}(\bm\Xi)
\label{eq:bi_fid2}
\end{align}
i.e., we assume the LF low-dimensional subspace is close to its HF counterpart. Here, $\boldsymbol{\delta}_r(\boldsymbol{\Xi})$ denotes the total error associated with this assumption, the KL expansion truncation, and finite sample approximation of $\vc_{j}^B$. Notice that in (\ref{eq:bi_fid}) and (\ref{eq:bi_fid2}), $\vc_j^B$ denotes estimates of the true BF coefficients given by exact projection or infinite HF samples. In particular, we find $\vc_{j}^B$ in an analogous manner to the least squares regression problem (\ref{eq:LS}), i.e.,
\begin{equation}
\underset{{\mC}^B}{\min} \;\; \Vert  {\mC}^B\boldsymbol{\eta}_n  - \mH_n \Vert_F,
\label{eq:LS_high}
\end{equation}
where $\mH_{n}\in\mathbb{R}^{M\times n}$ with $\mH_{n}(:,s)=\mH(:,j_s)$, $s=1,\dots,n$, denotes $n$ columns of $\mH$ that make up the randomly selected HF samples used for the BF estimate, and the reduced basis measurement matrix $\boldsymbol{\eta}_n \in \mathbb{R}^{r \times n}$ is such that $\boldsymbol{\eta}_n(j,i) \coloneqq \eta_{j}(\boldsymbol{\xi}_{i} )$. 
For ease of notation, we continue to denote the full $N$ sample HF matrix as $\mH$. We note that, following \cite{hampton2015coherence}, the regression problem (\ref{eq:LS_high}) requires $n\sim r\log(r)$ HF samples, which is considerably smaller than what is needed in (\ref{eq:LS}) when $r \ll P$. At this stage, useful statistics such as the expected value and variance of the QoI can be easily retrieved from the coefficients. 

Finally, we estimate the complete HF data $\mH$ through evaluating the BF basis at $N$ realizations and multiplying by the BF coefficients as
\begin{equation}
\widehat{\mH} = {\mC}^B \boldsymbol{\eta}_N,  
\label{eq:bi_approx}
\end{equation}
where $\widehat{\mH} \in \mathbb{R}^{M \times N}$ denotes an estimate of $\mH$ constructed from $n$ HF samples and $\boldsymbol{\eta}_N \in \mathbb{R}^{r \times N}$ is the BF basis evaluated for $N$ samples. The main steps of this SMR method are presented in Algorithm \ref{alg:sto_mod_red}.
\begin{algorithm}
\caption{Stochastic model reduction (SMR).}
\label{alg:sto_mod_red}
\begin{algorithmic}[1]
\item Construct the PC expansion (\ref{eq:pce_low}) from $N$ LF samples. 
\item Determine $r$, the size of reduced basis, from the eigenvalue decay of the covariance of LF QoI given by its PC approximation. 
\item Solve (\ref{eq:reduced_basis}) for the reduced basis  $\{ \eta_j^L (\boldsymbol{\Xi} )\}^r_{j=1}$ determined as described in Section \ref{sec:KL}.  
\item Solve (\ref{eq:LS_high}), using $n \sim r\log(r)$ HF samples, for the BF coefficients ${\mC}^B$. 
\item Estimate BF statistics readily from the BF coefficients, or follow (\ref{eq:bi_approx}) to construct a BF estimate of HF data $\mH$. 
\end{algorithmic}
\end{algorithm}
\subsection{Error estimates}
\label{sec:error_bound_analysis}

{\color{black}Having computed the BF solution $\vu^{B}(\bm\Xi)$, following Algorithm \ref{alg:sto_mod_red}, our goal is to derive estimates of of the errors ($\vu^H - \vu^B$) in a norm that we shall specify. In doing so, we provide an estimate of point-wise (in space) errors ($u_i^H - u_i^B$), $i=1,\dots, M$, where $u_i^H$ and $u_i^B$ are the $i$th entries of $\vu_i^H$ and $\vu_i^B$, respectively. Our approach builds upon two previous theorems, one related to the interpolation-based, reduced basis approach of \cite{narayan2014stochastic,hampton2018practical}, also employed in \cite{doostan2016bi, fairbanks2017low, skinner2019reduced,fairbanks2020bi}, and the other related to the PC sampling error developed in \cite{hampton2015coherence}}.

In the BF method of \cite{narayan2014stochastic,hampton2018practical}, we aim to learn an interpolation rule from $\mL$ that we can apply to $\mH$. To determine the interpolation rule we construct a rank $r \ll N$ matrix interpolative decomposition (MID) \cite{gu1996efficient, cheng2005compression, martinsson2011randomized} via column pivoted QR factorization as
\begin{align}
    \mL \mP & \approx \mQ \left[ \begin{array}{cc} \mR_{11} &\;  \mR_{12} \end{array} \right], \nonumber\\
    & = \mQ \mR_{11} \left[ \begin{array}{cc} \mI &\; \mR_{11}^\dagger \mR_{12} \end{array} \right], \nonumber
\end{align}
where $\mI$ is the $r\times r$ identity matrix,  $\mP \in \mathbb{R}^{N \times N}$ is a permutation matrix, $\mQ \in \mathbb{R}^{m \times r}$ has $r$ orthonormal columns, $\mR_{11} \in \mathbb{R}^{r \times r}$ is an upper triangular matrix, $\mR_{12} \in \mathbb{R}^{r \times (N - r)}$, and $\dagger$ denotes the pseudoinverse. Evidently, $\mQ \mR_{11}$ is equal to the left $r$ columns of $\mL \mP$, denoted by $\mL_r$ the column skeleton of $\mL$. 
The column indices of $\mL_r$, $i_s, \: s = 1, \dots, r$, correspond to random variable input samples $\{\boldsymbol{\xi_{i_s}} \}_{s=1}^r$ and LF QoI samples $\{\vu^L(\boldsymbol{\xi_{i_s}})\}_{s=1}^r$. To finalize the low-rank representation, we set $\bar{\mC}^L: = \left[ \begin{array}{cc} \mI & \;\mR_{11}^\dagger \mR_{12} \end{array} \right] \mP$ to be the interpolation coefficients and construct an estimate of the LF data as
\begin{equation}
\bar{\mL} \coloneqq \mL_r \bar{\mC}^L, 
\label{eq:L_est}
\end{equation}
which first selects $r$ columns of $\mL$, and second, interpolates these columns with the coefficients $\bar{\mC}^L$ to approximate $\mL \approx \bar{\mL}$ as a linear combination of $r$ basis vectors. 

The BF approximation is produced from sampling the HF model for the input samples $\{\boldsymbol{\xi_{i_s}} \}_{s=1}^r$ identified by the LF reduced basis, creating the HF column skeleton $\mH_r$. We note that in the SMR in Section \ref{sec:stochastic_model_reduction}, $H_n$ is comprised of $n$ random samples, while here with MID BF the $r$ samples are identified via the LF data. 
A BF approximation, $\bar{\mH}$, is determined as
\begin{equation}
\bar{\mH} \coloneqq \mH_r \bar{\mC}^L, 
\label{eq:MID_H}
\end{equation}
where $\mH_r$ is a set of $r$ columns of $\mH$, and $\bar{\mC}^L$ are the coefficients found from the LF data in equation (\ref{eq:L_est}). In this manner we have arrived at the MID BF estimate, $\bar{\mH}$, with a total of $N$ LF samples and $r$ HF samples. 
Next, we introduce Theorem~\ref{thm:bifi} from \cite{hampton2018practical}.
\begin{theorem}
\label{thm:bifi} (Theorem~1 of \cite{hampton2018practical})
For any $\tau \ge 0$, let 
\begin{align}
\epsilon(\tau) = \| \mH^T\mH - \tau \mL^T\mL \|_2.
\label{eq:epsdef}
\end{align}
Let $\bar{\mH}$ and $\bar{\mL}$ be corresponding static coefficient BF estimates of rank $r$ with coefficients $\bar{\mC}^L$, and let $\sigma_k$ denote the $k$th largest singular value of $\mL$. Then,
\begin{align}
 \|\mH - \bar{\mH} \|_2 & \le \rho_k(\tau);
 \nonumber
\end{align}
where $\rho_k(\tau)$ is defined by
\begin{align}
\rho_k(\tau) \coloneqq& \mathop{\min}\limits_{\tau,k \le \mbox{rank}(\mL)} (1+\|\bar{\mC}^L\|_2) \sqrt{\tau \sigma_{k+1}^2 + \epsilon(\tau)}\nonumber \\
& + \|\mL-\bar{\mL}\|_2 \sqrt{\tau + \epsilon(\tau)\sigma_k^{-2}}.
\label{eq:thm1:rho}
\end{align}
When $k = rank(\mL)$, we set $\sigma_{k+1}=0$.
\end{theorem}
\begin{remark}
We note that this theorem is better suited for the MID \cite{hampton2018practical} than for the SMR BF algorithm of this work. However, this bound is relevant in this case as well. Of practical relevance, this bound is identified using predominantly LF samples, with a limited number of HF samples to estimate (\ref{eq:epsdef}).
\end{remark}

The second theorem that we build upon concerns the number of samples needed to accurately recover coefficients $\mC^B$ from (\ref{eq:LS_high}). From \cite{hampton2015compressive,hampton2015coherence}, we recall the coherence parameter $\mu$ of  that provides a bound on the realized spectral radius of $\boldsymbol{\eta}$ as
\begin{equation}
\mu \coloneqq \sup_{\boldsymbol{\xi} \in \Omega} \sum_{j=1}^r \vert \eta_{j} (\boldsymbol{\xi}) \vert^2. 
\label{eq:coherence}
\end{equation}
We note that under coherence optimal conditions \cite{hampton2015coherence} $\mu = r$, and such conditions can be guaranteed by importance sampling. 

We now present Theorem~\ref{thm:ez_pce} that uses this coherence parameter to bound the number of samples, $n$, necessary to recover coefficients ${\mC}^B$. In Theorem~\ref{thm:ez_pce}, we revisit the truncation error of equation (\ref{eq:bi_fid2}), denoted for a given spatial point as $\delta_{r,i}(\boldsymbol{\Xi})$.
\begin{theorem}
\label{thm:ez_pce} (Theorem~2.1 of \cite{hampton2015coherence})
Let
\begin{align}
{u}_{i}^B(\bm{\Xi}) = \mathop{\sum}\limits_{j=1}^r {\mC}^B (i,j)\eta_{j}(\bm{\Xi}),
\nonumber
\end{align}
where ${\mC}^B$ is the least squares solution to (\ref{eq:LS_high}). It follows that for $\mathcal{E}$, which is independent of $i$, and is a sampling event that occurs with probability
\begin{align}
\label{eq:rare_prob}
\mathbb{P}(\mathcal{E}) \ge 1-2r \exp(-0.1 n \mu^{-1}),
\end{align}
that
\begin{align}
\label{eq:small_nu_bound}
 \nu_{i} &\coloneqq \mathbb{E}\left(\|u_{i}^H(\bm{\Xi})-u_{i}^B(\bm{\Xi})\|^2_{L_2(\Omega,f)};\mathcal{E}\right)\nonumber \\
 &\le \left(1+\frac{4\mu}{n}\right)\mathbb{E}(\delta^{2}_{r,i}(\bm{\Xi})),
\end{align}
where $\mu$ is as in (\ref{eq:coherence}), and
\begin{align}
\mathbb{E}\left(X;\mathcal{E}\right) = \int_{\mathcal{E}}X(\bm{\xi})f(\bm{\xi})d\bm{\xi} = \mathbb{E}(X|\mathcal{E})\mathbb{P}(\mathcal{E})
\nonumber
\end{align}
denotes the expectation restricted to the event (also known as restricted expectation), and is closely related to conditional expectation.
\end{theorem}
\begin{remark}
 We note that this error bound applies pointwise in space. The error from the BF approximation can be concentrated in certain spatial regions, and this often applies to the PC approximations for these $M$ points as well. Summing the $\nu_i$ over all points in space allows for some simplifications, as we shall see in Corollary~\ref{cor:cor1} and Corollary~\ref{cor:cor2}.
Additionally, the probability estimate in (\ref{eq:rare_prob}) is a significant underestimate for small sample sizes. 
\end{remark}

We now show that Theorem~\ref{thm:bifi} can be utilized for the SMR case for an \textit{a priori} error consideration, at least in cases where the MID does not fully recover $\mH$. This is assured by a mild assumption on the relationship between $\mL$ and $\mH$.

\begin{corollary}
\label{cor:cor1}
Assume that there does not exist any matrix $\mF\in\mathbb{R}^{M\times m}$ with singular values restricted to $\{0,1\}$ and a constant $q$ such that
\begin{align}
 q\mF\mL &= \mH.
 \label{eq:q_constraint}
\end{align}
Assuming also the conditions of Theorem~\ref{thm:ez_pce}, it follows that,
\begin{align}
 \label{eq:nu_bound_points_tot}
 \nu_{i} &= \left(1+\frac{4\mu}{n}\right)\left(\frac{\zeta_{i}}{N}\rho^2_k(\tau)\right);\\
 \label{eq:nu_bound_tot}
 \mathop{\sum}\limits_{i=1}^{M}\nu_{i} &= \left(1+\frac{4\mu}{n}\right)\left(\frac{R\bar{\zeta}}{N}\rho^2_k(\tau)\right).
\end{align}
Here, $\rho_k(\tau)$ is as in Theorem~\ref{thm:bifi}, $\zeta_{i}$ and $\bar{\zeta}$ are random variables which converge a.s. to finite values $g_{r,i}$ and $g_r$, respectively, $R\le M$ is the rank of $\mH-\bar{\mH}$, and the rest is as in Theorem~\ref{thm:ez_pce}.
\end{corollary}
%
%
\begin{proof}
If $\mathbb{E}(\delta^{2}_{r,i}(\bm{\Xi}))=0$, then $\nu_i=0$, and there is nothing to prove; hence, assume this is not the case. Implicitly define $\zeta_{i}$ such that 
\begin{align}
\mathbb{E}(\delta^{2}_{r,i}(\bm{\Xi})) &= \zeta_{i}\frac{\rho^2_k(\tau)}{N}.
\label{eq:zeta}
\end{align}
$\zeta_{i}$ is defined unless $\rho^2_k(\tau)=0$. If this occurs at $\tau=0$, then $\mH$ must be the zero matrix, which is a degenerate case for which the bounds hold; therefore, consider the case that $\tau>0$. In particular, $\rho^2_k(\tau)=0$ implies 
\begin{align*}
\mH^T\mH & = \tau\mL^T\mL. \nonumber
\end{align*}
This implies the existence of a matrix $\mF\in \mathbb{R}^{M\times m}$ whose singular values are all in $\{0,1\}$ satisfying (\ref{eq:q_constraint}) with $q=\sqrt{\tau}$. Hence, assuming that such a pair of $\mF$ and $\tau$ does not exist implies that each $\zeta_{i}>0$ is well defined. The LHS of (\ref{eq:zeta}) depends only on $r$. In contrast, the RHS depends on $N$ and $n$, but as $N\rightarrow\infty$ and $n\rightarrow\infty$, the RHS converges to a non-zero limiting value. This, with Theorem~\ref{thm:ez_pce}, is sufficient to show (\ref{eq:nu_bound_points_tot}).

To show (\ref{eq:nu_bound_tot}), we note that 
from Theorem~\ref{thm:ez_pce}, it follows that for some $\bar{a}$ whose limiting value is in $(0,1]$,
\begin{align}
\mathop{\sum}\limits_{i=1}^M\nu_i & =\bar{a}\left(1+\frac{4\mu}{n}\right)\mathop{\sum}\limits_{i=1}^M\mathbb{E}(\delta^2_{r,i}(\bm{\Xi})).
\label{eq:LHScor1A}
\end{align}
Consider $b>0$ such that
\begin{align}
\mathop{\sum}\limits_{i=1}^M\mathbb{E}(\delta^2_{r,i}(\bm{\Xi})) & = \frac{R b}{N}\rho^2_k(\tau),
\label{eq:LHScor1B}
\end{align}
where $R$ is the rank of $\mH - \bar{\mH}$. Note that the scaling $R/N$ arises from Theorem~\ref{thm:bifi}, and the fact that
\begin{align*}
R\|\mH-\bar{\mH}\|^2_2 \ge \|\mH-\bar{\mH}\|^2_F.
\end{align*}
Under the same conditions as above for (\ref{eq:zeta}), noting that $R\le M$, $b$ is finite and converges to a non-zero limit as $N\rightarrow\infty$. Combining (\ref{eq:LHScor1A}) and (\ref{eq:LHScor1B}),
\begin{align*}
\mathop{\sum}\limits_{i=1}^{M}\nu_i &= \left(1+\frac{4\mu}{n}\right)\frac{R\bar{a}b}{N}\rho^2_k(\tau).
\end{align*}
Setting $\bar{\zeta} = \bar{a} b$ shows (\ref{eq:nu_bound_tot}).
$\blacksquare$
\end{proof}

Observe in the RHS of (\ref{eq:nu_bound_tot}) that the first term is consistent with (\ref{eq:nu_bound_points_tot}) and converges to $1$ as $n \rightarrow \infty$. The second term converges to a constant as $\rho^2_k (\tau)$ grows with $N$. 
An in depth analysis of $\rho^2_k (\tau)$ can be found in \cite{hampton2018practical}. 

{\color{black}
\begin{remark}
Estimating $\zeta_i$ and $\bar{\zeta}$ in Corollary \ref{cor:cor1} requires knowing the true error $\delta_{r}(\bm\Xi)$ which is not available nor it is practically realistic to generate. Therefore, the utility of the results in (\ref{eq:nu_bound_points_tot}) and (\ref{eq:nu_bound_tot}) is to provide insight into the convergence of the SMR approach with respect to the decay of the spectrum of the low-fidelity matrix $\bm{L}$ and the proximity of the low- and high-fidelity Gramian matrices -- measured by $\epsilon(\tau)$ -- through $\rho_k(\tau)$.
\end{remark}
}

{\color{black}
\begin{remark}
The choice of $r$ is of practical importance, as increasing the number of terms in the KL expansion is a natural way to improve the BF estimate. Theorem~\ref{thm:ez_pce} and Corollary~\ref{cor:cor1} show that the number of samples grows as the coherence $\mu$. As $\mu$ is often of order $r$, this suggests $n$ can be of order $r\log(r)$~\cite{cohen2013stability,hampton2015coherence}. Increasing $r$ can only reduce error to a point, as it is the LF model which determines the KL expansion that yields the PC approximation. Theorem~\ref{thm:bifi} and Corollary~\ref{cor:cor1} suggest that under a mild assumption this dependence on the LF model cannot be much worse than the dependence in the MID case, which is seen to be robust with regards to differences between the LF and HF models as shown in \cite{narayan2014stochastic,zhu2014computational,doostan2016bi,hampton2018practical,skinner2019reduced,fairbanks2020bi}.
\end{remark}
}

\subsubsection{Practical bounds via moments}
\label{sec:pract_enhance}

We now present Theorem~\ref{thm:new_pc}, which provides a straightforward estimate of $\nu_i$, useful for practical \textit{a posteriori} error estimation. {\color{black}In our practical estimates the number of HF samples used to compute the bound is denoted by $\hat{n}\ge n$.} We note that the notation $\mH_{{\color{black}\hat{n}}}(i,:)$ refers to the $i$th row of $\mH_{{\color{black}\hat{n}}} \in \mathbb{R}^{M\times {\color{black}\hat{n}}}$.

\begin{theorem}
\label{thm:new_pc}
For $\mathcal{E}$ as in Theorem~\ref{thm:ez_pce}, assumed to be satisfied, it follows that,
\begin{align}
 \nu_{i} :&= \mathbb{E}\left(\|u_{i}^H(\bm{\Xi})-{u}_{i}^B(\bm{\Xi})\|^2_{L_2(\Omega,f)};\mathcal{E}\right) \nonumber\\
 \label{eq:nu_bound_theta_i}
      &= \frac{\theta_{ i}}{{\color{black}\hat{n}}}\|\mH_{{\color{black}\hat{n}}}(i,:)-\widehat{\mH}_{{\color{black}\hat{n}}}(i,:)\|^2_2;\\
 \label{eq:nu_bound_theta}
 \mathop{\sum}\limits_{i=1}^M\nu_i &= \frac{\bar{\theta}}{{\color{black}\hat{n}}}\|\mH_{{\color{black}\hat{n}}}-\widehat{\mH}_{{\color{black}\hat{n}}}\|^2_F.
\end{align}
Here, $\mH_{{\color{black}\hat{n}}} \in \mathbb{R}^{M \times {\color{black}\hat{n}}}$ is the matrix of ${\color{black}\hat{n}}$ HF samples, $\widehat{\mH}_{{\color{black}\hat{n}}}$, which is the BF approximation to $\mH_{{\color{black}\hat{n}}}$, and $\theta_{i}$ and $\bar{\theta}$ are random variables {\color{black}that converge almost surely to $1$ as ${\color{black}\hat{n}} \rightarrow \infty$}.
\end{theorem}

\begin{proof}
We note that $\theta_{i}$ and $\bar\theta$ depend on the $N$ HF random samples used to estimate the error; that is, $\theta_{i}$ and $\bar\theta$ are multiplicative factors to correct the sample mean to the true mean. Recall that, $\mH(i,k)$ contains $u_{i}^H(\bm{\xi}_k)$, and that $\widehat{\mH}(i,k)$ contains the corresponding $u^B_{i}(\bm{\xi}_k)$ computed via SMR. 
When the sampling event $\mathcal{E}$ occurs, as assumed in Theorem~\ref{thm:ez_pce}, it follows that, (\ref{eq:nu_bound_theta_i}) is an unbiased, consistent estimate for $\nu_i$. By the strong law of large numbers $\theta_{i}$ converges to $1$ almost surely as $\hat{n} \rightarrow \infty$. The analogous argument shows (\ref{eq:nu_bound_theta}). $\blacksquare$\\
\end{proof}
{\color{black}
 In Corollary~\ref{cor:cor2}, below, we remove the random variables $\theta_{i}$ and $\bar{\theta}$ from (\ref{eq:nu_bound_theta_i}) and (\ref{eq:nu_bound_theta}) by bounding them probabilistically}. To proceed, we define two random variables to which the Berry-Esseen Theorem \cite{berry1941accuracy, shevtsova2011absolute} will be applied. Estimating the moments of these random variables will correspond to bounds of the error from samples for a particular problem.

\begin{align}
\label{eq:V_def}
\mV_{i,j} &= |\mH_{{\color{black}\hat{n}}}(i,j)-\widehat{\mH}_{{\color{black}\hat{n}}}(i,j)|^2; & \mW_{j} &= \mathop{\sum}\limits_{i=1}^M \mV_{i,j}.
\end{align}
We note that $\{\mV_{i,j}\}_{j=1}^{{\color{black}\hat{n}}}$ and $\{\mW_{j}\}_{j=1}^{{\color{black}\hat{n}}}$ are independent and identically distributed for each $i$. We define the moments of  $\mV_{i,j}$ and $\mW_{j}$ as
\begin{align}
\label{eq:moments1_def}
\alpha_{\mV_{i}}   & = \mathbb{E}(\mV_{i,j}); & \alpha_{\mW}       & = \mathbb{E}(\mW_{j});\\
\label{eq:moments2_def}
\beta^2_{\mV_{i}} & = \mathbb{E}|\mV_{i,j} - \alpha_{\mV_{i}}|^2; & \beta^2_{\mW} & = \mathbb{E}|\mW_{j} - \alpha_{\mW}|^2;\\
\label{eq:moments3_def}
\gamma_{\mV_{i}}     & = \mathbb{E}|\mV_{i,j} - \alpha_{\mV_{i}}|^3; & \gamma_{\mW}     & = \mathbb{E}|\mW_{j} - \alpha_{\mW} |^3.
\end{align}
Next, we define the appropriately normalized random variables as 
\begin{align}
\label{eq:tilde_def}
\tilde{\mV}_{i} &= {\color{black}\hat{n}}^{-1/2} \mathop{\sum}\limits_{j=1}^{{\color{black}\hat{n}}} \frac{\mV_{i,j} - \alpha_{\mV_{i}}}{\beta_{\mV_{i}}};
& \tilde{\mW} &= {\color{black}\hat{n}}^{-1/2} \mathop{\sum}\limits_{j=1}^{{\color{black}\hat{n}}}\frac{\mW_{j} - \alpha_{\mW}}{\beta_{\mW}},
\end{align}
to which we apply the Berry-Esseen Theorem, restated in Theorem~\ref{thm:berry_Esseen}.{\color{black} The Berry-Esseen approach reduces the identification of a bound to that of identifying the first three centralized moments (\ref{eq:moments1_def})-(\ref{eq:moments3_def}), which can be estimated from QoI samples.}

\begin{theorem}
\label{thm:berry_Esseen}\cite{berry1941accuracy,shevtsova2011absolute}
 Let $F_{\tilde{\mV}_{i}}(\cdot)$ be the cumulative distribution for $\tilde{\mV}_{i}$ and $\Phi(\cdot)$ the cumulative distribution function for the standard normal random variable. There exists a positive constant $C\le 0.4748$ such that for all $t$ ,
 \begin{align}
 |F_{\tilde{\mV}_{i}}(t) - \Phi(t)| &\le \frac{C\gamma_{\mV_{i}}}{\beta^3_{\mV_{i}}\sqrt{{\color{black}\hat{n}}}}; &  |F_{\tilde{\mW}}(t) - \Phi(t)| &\le \frac{C\gamma_{\mW}}{\beta^3_{\mW} \sqrt{{\color{black}\hat{n}}}}. \nonumber
 \end{align} 
\end{theorem}

\begin{corollary}
\label{cor:cor2}
Under the conditions of Theorem~\ref{thm:new_pc}, it follows that, for each $i$, any $t$, and with a probability
\begin{equation}
p_i(t) \ge \Phi(t) - \frac{C\gamma_{\mV_{i}}}{\beta^3_{\mV_{i}}\sqrt{{\color{black}\hat{n}}}},
\label{eq:prob_vector}
\end{equation}
the following bound holds;
\begin{equation}
 \nu_{i} \le \left(\alpha_{\mV_{i}} + \frac{t\beta_{\mV_{i}}}{\sqrt{{\color{black}\hat{n}}}}\right).
 \label{eq:nu_bound_i_pract}
\end{equation}
Similarly, for any $t$, and with a probability
\begin{equation}
p(t) \ge \Phi(t) - \frac{C\gamma_{\mW}}{\beta^3_{\mW}\sqrt{{\color{black}\hat{n}}}},
\label{eq:prob_sum}
\end{equation}
the following bound holds;
\begin{equation}
 \mathop{\sum}\limits_{i=1}^{M}\nu_{i} \le \left(\alpha_{\mW} + \frac{t\beta_{\mW}}{\sqrt{{\color{black}{\color{black}\hat{n}}}}}\right).
 \label{eq:nu_bound_pract}
\end{equation}
\end{corollary}
\begin{proof}
 This proof follows directly from algebra on the Berry-Esseen bound, (\ref{eq:nu_bound_theta_i}), and (\ref{eq:nu_bound_theta}), using the definitions of (\ref{eq:tilde_def}). $\blacksquare$
\end{proof}

{\color{black}

It is worthwhile highlighting that the approach of Corollary \ref{cor:cor2} differs from cross-validation in two key respects. First, cross-validation uses a subset of samples to compute an estimate, and another subset to estimate the error. This is especially useful when an estimate tends to overfit the data. Here, we oversample when computing our estimates, and thus avoid the problems of overfitting. Second, the Berry-Esseen approach can compute a single estimate, and all of the data to compute the error with a reasonably accurate probablistic bound on $\nu_i$. This cannot be easily translated to cross-validation. Each $\nu_i$ depends on the samples used to construct it, and cross-validation depends on computing different estimates from different samples. In this way the $\nu_i$ would differ, and as each construction would share a number of samples, they would not be independent. Averaging the various error estimates with cross-validation, provides an estimate for an unintuitive quantity.} 

Algorithm \ref{alg:bound_pract} presents the method of evaluating the practical error bounds in (\ref{eq:nu_bound_i_pract}) and (\ref{eq:nu_bound_pract}). 
\begin{algorithm}
\caption{Practical error bounds of (\ref{eq:nu_bound_i_pract}) and (\ref{eq:nu_bound_pract}).}
\label{alg:bound_pract}
\begin{algorithmic}[1]
\item Choose $t$ and corresponding $\Phi(t)$ from the standard normal distribution.
\item Use $\hat{n}$ HF samples to evaluate the variables in (\ref{eq:V_def}) and their moments in (\ref{eq:moments1_def}) - (\ref{eq:moments3_def}). 
\item Evaluate the probability and associated error bound  in equations (\ref{eq:prob_vector}) and (\ref{eq:nu_bound_i_pract}) for pointwise estimates, or equations (\ref{eq:prob_sum}) and (\ref{eq:nu_bound_pract}) for a sum of all points.
\end{algorithmic}
\end{algorithm}

\section{Numerical examples}
\label{sec:numerical_examples_pc}

In this section we demonstrate the effectiveness of the reduced basis approach and practical error bound described in Section \ref{sec:method_detail_pc} (Corollary \ref{cor:cor2}) on three numerical examples from fluid mechanics. The examples model a range of Reynolds number (Re) with uncertainty in initial and boundary conditions, fully turbulent heat transfer with uncertain boundary conditions{\color{black},} and fully turbulent flow featuring geometric uncertainty. In each application, we are interested in estimating an output QoI subject to stochastic inputs. 
We compare the performance of estimates arrived at with LF, HF, and BF methods, implementing the BF Algorithm \ref{alg:sto_mod_red}. The LF method utilizes $N$ less expensive samples, while the HF method is made up of $n \ll N$ expensive HF samples. It is assumed that the computational cost of the HF model is much greater than its LF counterpart. 
The BF approach utilizes $n$ HF samples and $N$ LF samples. In this manner, the BF and HF methods use the same number of HF samples and, provided the LF model cost is negligible, have approximately equal computational cost. We compare model estimates of the QoIs mean and variance that can be computed from the expansion coefficients as $\mathbb{E}[u]  = c_1 $ and $\text{Var}[u] = \sum_{j=2}^P  c_{j}^2$. We define the reference solution as the full set of HF data, denoted by \textit{ref}, and calculate the relative 2-norm error of of our vector QoI estimates as
\begin{align}
    \text{Relative Error}_{\mathbb{E}[\vu]} &= \frac{\Vert \mathbb{E}[\vu]_{ref} - \mathbb{E}[\vu] \Vert_2}{\Vert \mathbb{E}[\vu]_{ref} \Vert_2}; \\
    \text{Relative Error}_{\text{Var}[\vu]} &= \frac{\Vert \text{Var}[\vu]_{ref} - \text{Var}[\vu] \Vert_2}{\Vert \text{Var}[\vu]_{ref} \Vert_2}  \label{eq:err_metric}.
\end{align}

During the numerical examples, we employ compressive sensing for solving the $\ell_1$-minimization problem (\ref{eq:ell_1}) to determine LF PC expansion coefficients in equation (\ref{eq:pce_low}) and use least squares to solve for the BF coefficients of equation (\ref{eq:bi_fid}). In each numerical example, we investigate the utility of the practical bounds (\ref{eq:nu_bound_i_pract}) and (\ref{eq:nu_bound_pract}) implemented via Algorithm \ref{alg:bound_pract}. 
In doing so, for all numerical test cases, we set $t = 2.0$ in Corollary~\ref{cor:cor2} {\color{black}and use the same $n$ HF samples used to build the BF approximation to compute the bound, i.e., $\hat{n}=n$.} 

\subsection{Example I: lid-driven cavity flow}
\label{sec:case_1_ldc}

We consider the benchmark problem of a lid-driven cavity flow \cite{ghia1982high}, where the QoI is the $y$ velocity component measured along the line $y = 0.5$, $0 \leq x \leq 1.0$; see Figure \ref{fig:LDC_1} (a). The problem addresses the two-dimensional, steady, incompressible Navier-Stokes equations with nominal $Re = 10^2$ and is solved for in \texttt{FEniCS} using Taylor-Hood elements \cite{AlnaesBlechta2015a} with codes based on \cite{chandGithub}.

The geometry is a unit square with Dirichlet boundary conditions on all four walls as depicted in Figure \ref{fig:LDC_1} (a). The top wall is moving at a fixed velocity, while the remaining three are stationary.
The stochastic dimension of the problem is two; the velocity of the top plate and the kinematic viscosity of the fluid are treated as uniform random variables within intervals $ [0.8 \; 1.2]$ and $ [0.009 \; 0.011]$, respectively. 

A coarse $4\times 4$ grid mesh is used as the LF model and a fine $64 \times 64$ grid mesh is used as the HF model, depicted in Figures \ref{fig:LDC_1} (b) and (c){\color{black},} respectively. Grid points are more closely packed with proximity to walls to improve the resolution of flow gradients. The HF model's computational expense is approximately 155 times greater than its LF counterpart. 

\begin{figure*}[ht!]
\centering
\begin{subfigure}{0.3\textwidth}
  \centering
  \includegraphics[width=\textwidth]{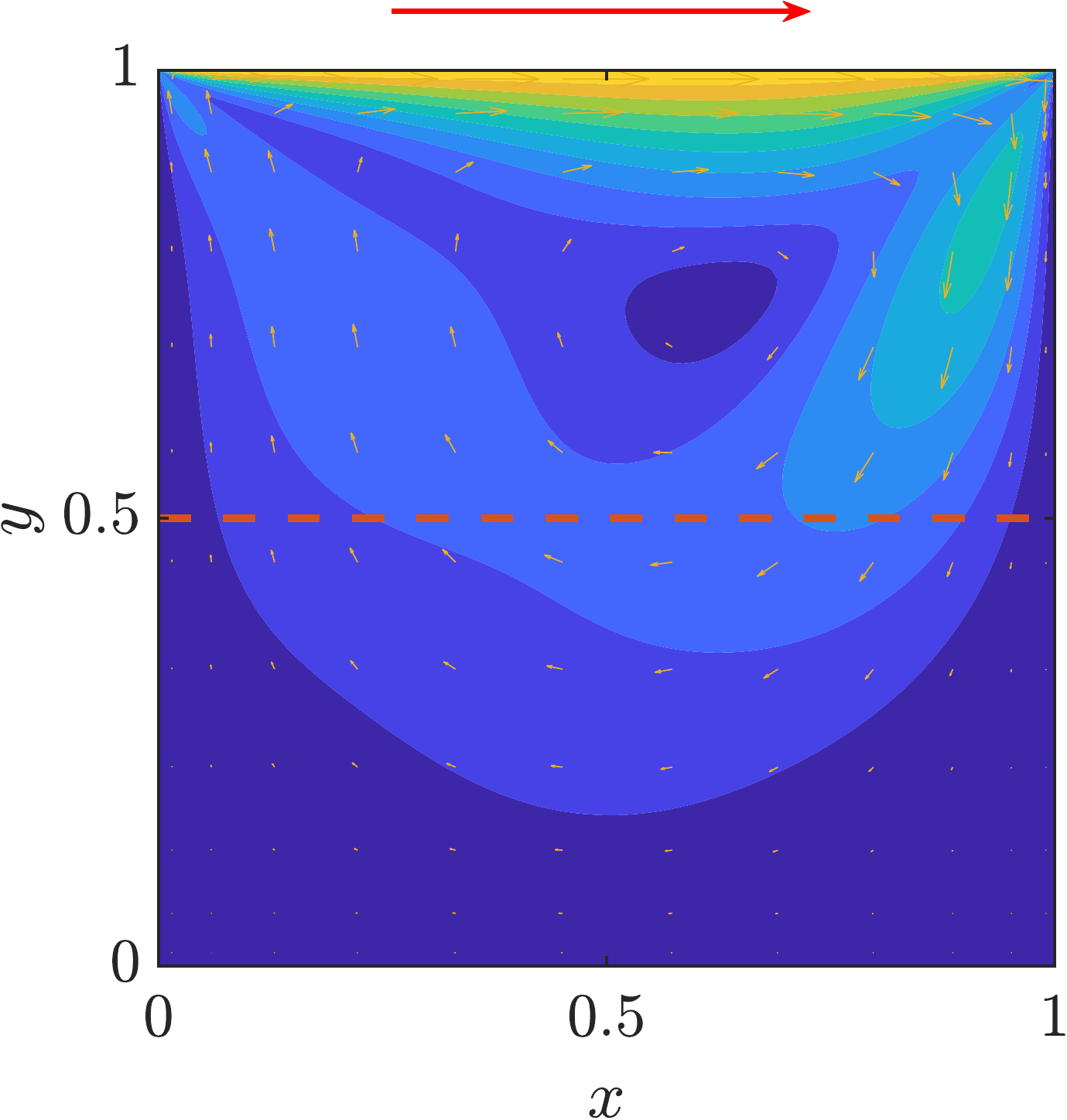}
  \caption{Cavity geometry.}
  \label{fig:LDC_geom_pc}
\end{subfigure}
\hfill
\begin{subfigure}{0.3\textwidth}
  \centering
  \includegraphics[width=\textwidth]{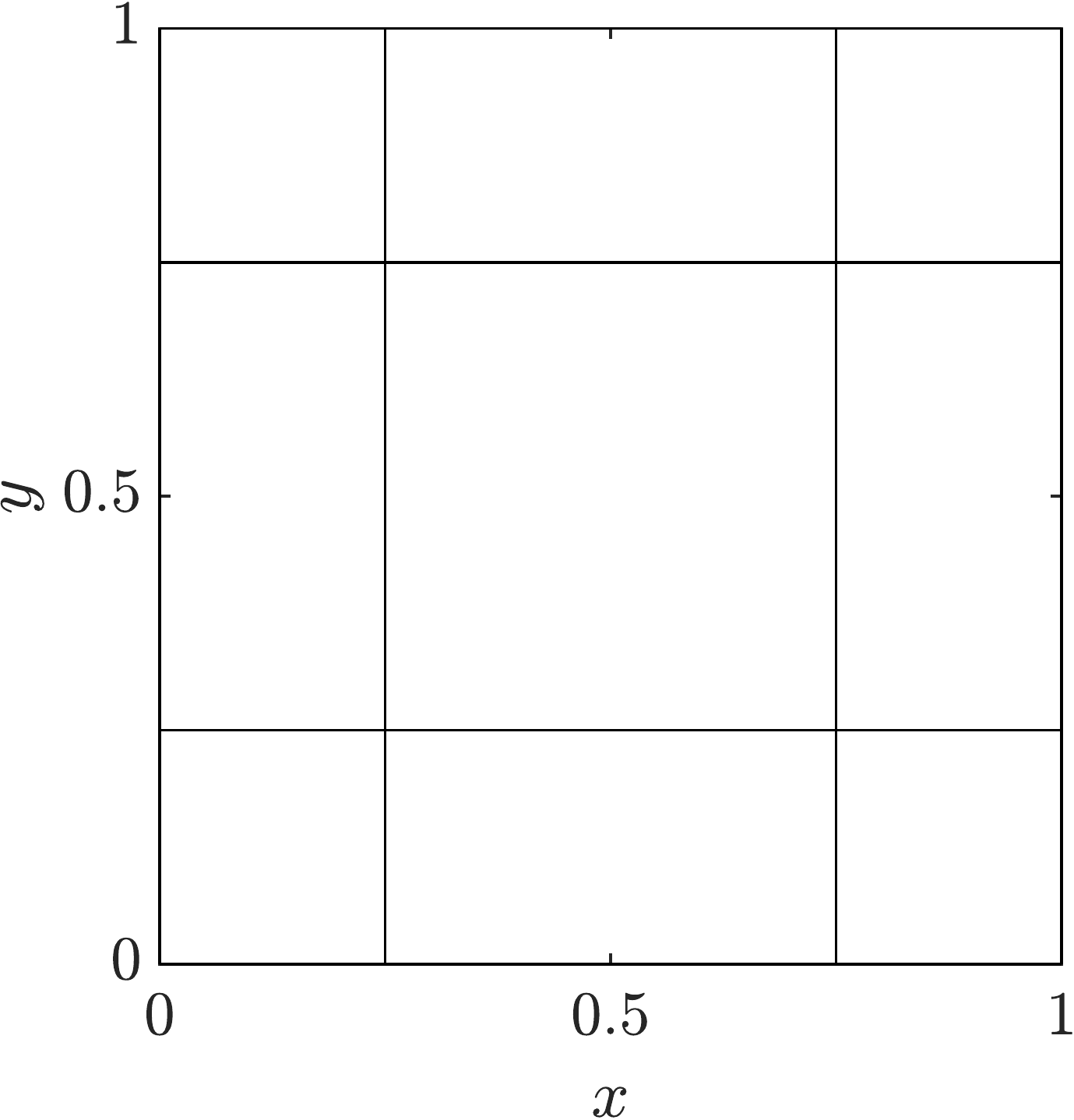}
  \caption{LF mesh.}
  \label{fig:LDC_meshL_pc}
\end{subfigure}
\hfill
\begin{subfigure}{0.3\textwidth}
  \centering
  \includegraphics[width=\textwidth]{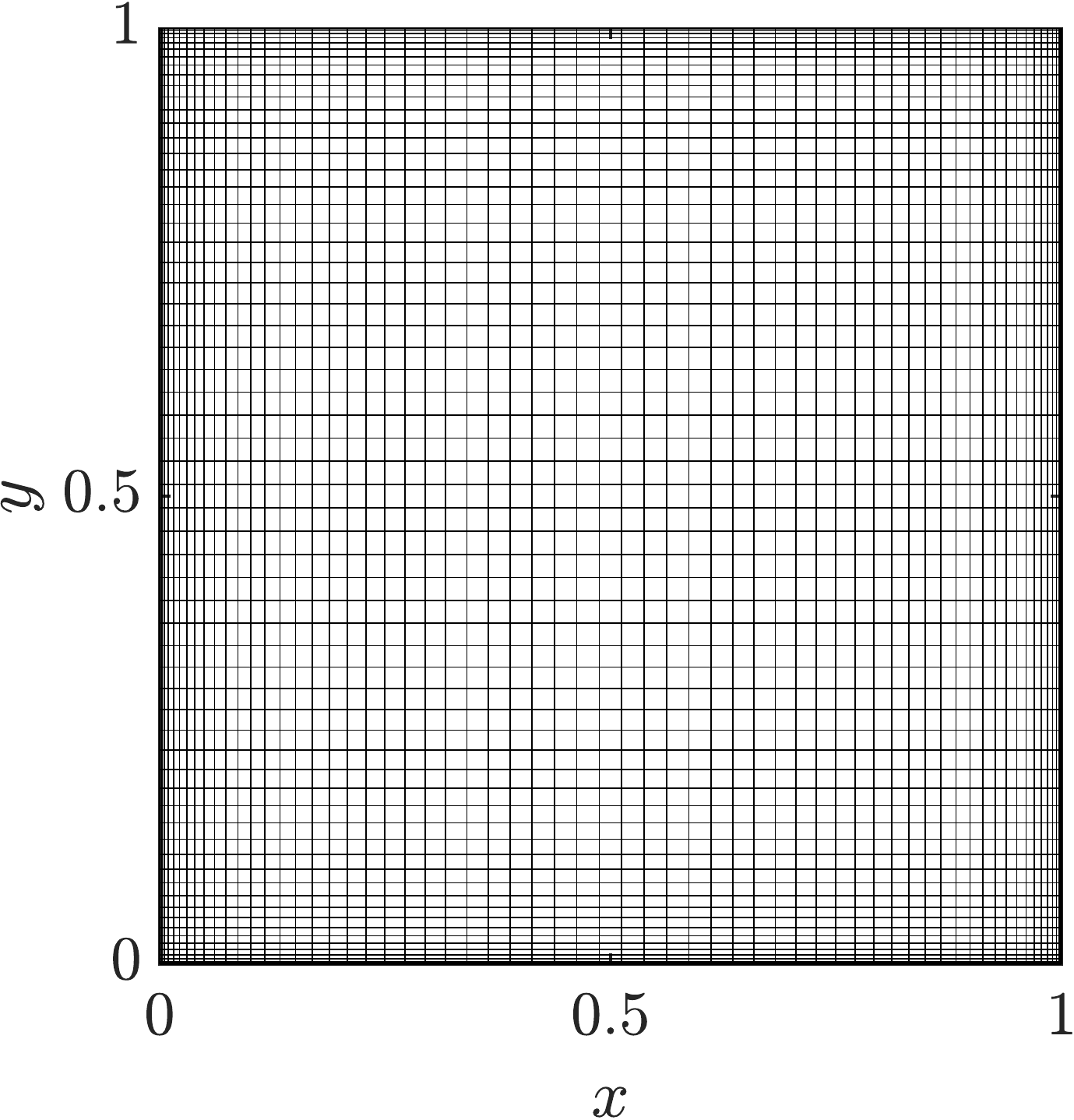}
  \caption{HF mesh.}
  \label{fig:LDC_meshH_pc}
\end{subfigure}
\caption[Lid-driven cavity geometry of Example I.]{Lid-driven cavity geometry of Example I. QoI is the vertical velocity along the line $y = 0.5$, $0 \leq x \leq 1$.}
\label{fig:LDC_1}
\end{figure*}
\subsubsection{Results} 
\label{sec:ldc_results_pc}

We first address the average BF performance over $100$ repetitions as the number of HF samples and approximation rank are varied, each computed from independent sets of randomly chosen $n$ HF samples.   Figure \ref{fig:LDC_BR} depicts the average relative error of a PC expansion of order $p= 4$ constructed from $N=200$ LF samples, $n$ HF samples, and a BF estimate that utilizes $n$ HF samples. The reference solution is calculated via a PC expansion with $N=200$ HF samples. BF estimates for approximation ranks $r=2$, $4${\color{black},} and $7$ are reported.  
\begin{figure*}[ht!]
\centering
\includegraphics[width=\textwidth]{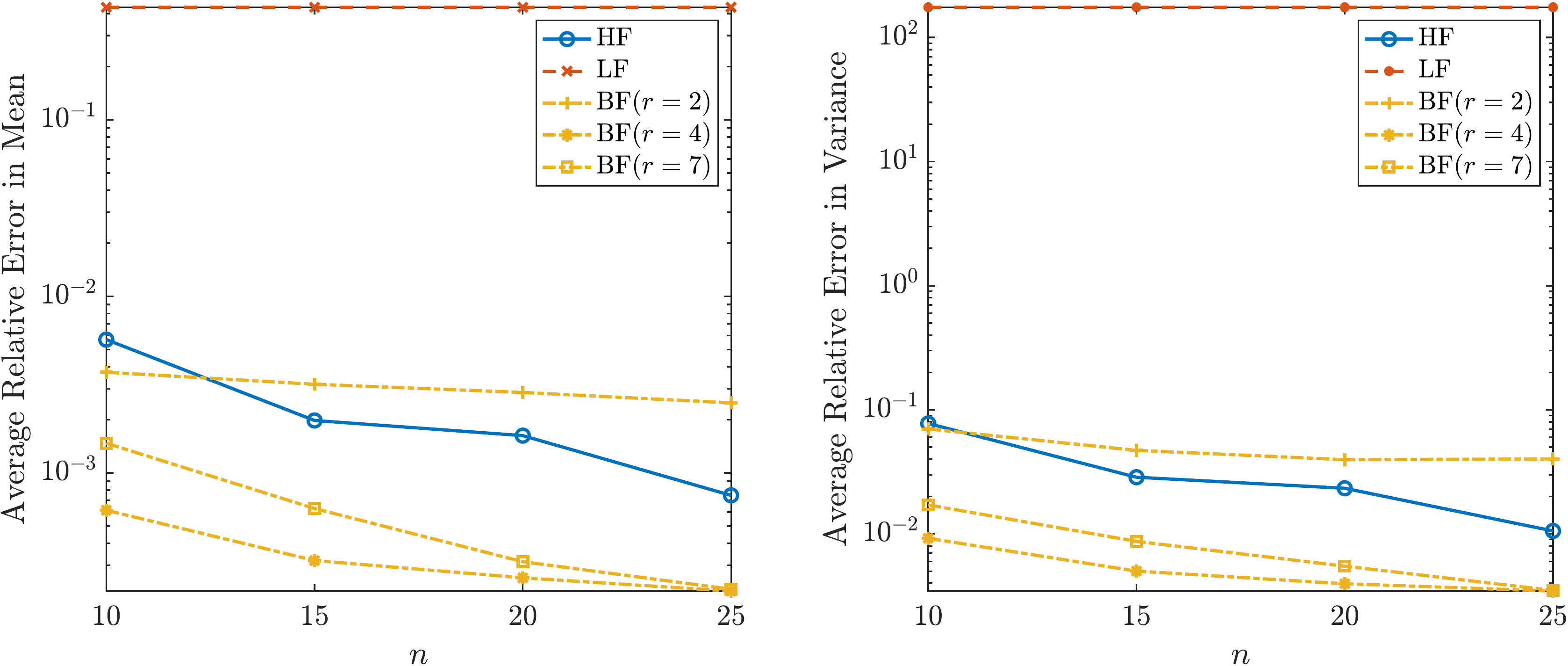}
\caption[The average relative error for (left) mean and (right) variance of the vertical velocity component through a half of the cavity in Example I.]{The average relative error for (left) mean and (right) variance of the vertical velocity component through a half of the cavity in Example I. Plotted are the HF, LF, and BF estimates for approximation rank $r$. The average is calculated from $100$ repetitions.}
\label{fig:LDC_BR}	
\end{figure*}
Evident in Figure \ref{fig:LDC_BR} is that the LF model provides a poor estimate of both the mean and variance. This is expected given the aggressive coarseness of the LF solution and its corresponding failure to accurately capture the problem physics. A reduced basis of rank $r=4$ accurately estimates the QoI mean and variance with $n \geq 10$ HF samples available, achieving an order of magnitude better performance than the HF model with $n=10$. Proceeding with a reduced basis of rank $r=4$, Figure \ref{fig:LDC_mean_var} compares the mean and variance estimates made with $n=10$ HF samples. The BF estimate accurately captures both the mean and variance of the vertical velocity component across the cavity. The HF estimate exhibits comparable accuracy in estimation of the mean, but is exceeded by the BF estimation of the variance, a more challenging statistic to approximate. The full PC expansion has $P = 15$ terms, reduced to $r=4$ in the BF approach. 
\begin{figure*}[ht!]
\centering
\includegraphics[width=\textwidth]{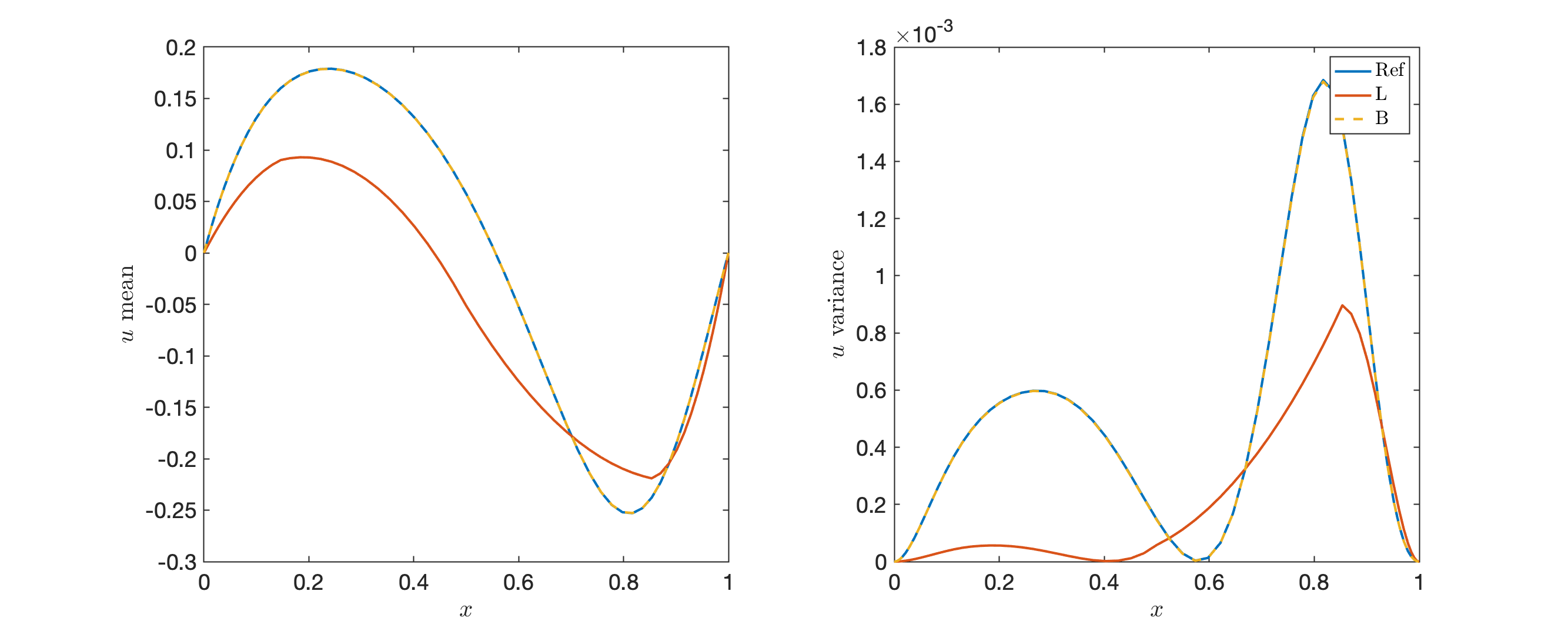}
\caption[Vertical velocity component along the line $y=0.5$ in Example I. Shown are estimates of the (left) mean and (right) variance, for rank $r = 4$ with $n=10$ HF samples.]{Vertical velocity component along the line $y=0.5$ in Example I. Shown are estimates of the (left) mean and (right) variance, for rank $r = 4$ with $n=10$ HF samples. The reference solution is given by $Ref$ alongside the HF, LF and BF estimates.}
\label{fig:LDC_mean_var}	
\end{figure*}

Figure \ref{fig:LDC_Eig} presents the normalized eigenvalues of the LF and reference covariance matrices. The swift decay in eigenvalues of both LF and reference solutions suggest this problem is amenable to reduced basis approximation. This agrees with the reduced basis approximations of Figures \ref{fig:LDC_BR} and \ref{fig:LDC_mean_var}, where only the first four basis functions are needed to accurately capture the QoI. 
\begin{figure}[ht!]
\centering
\includegraphics[width=\columnwidth]{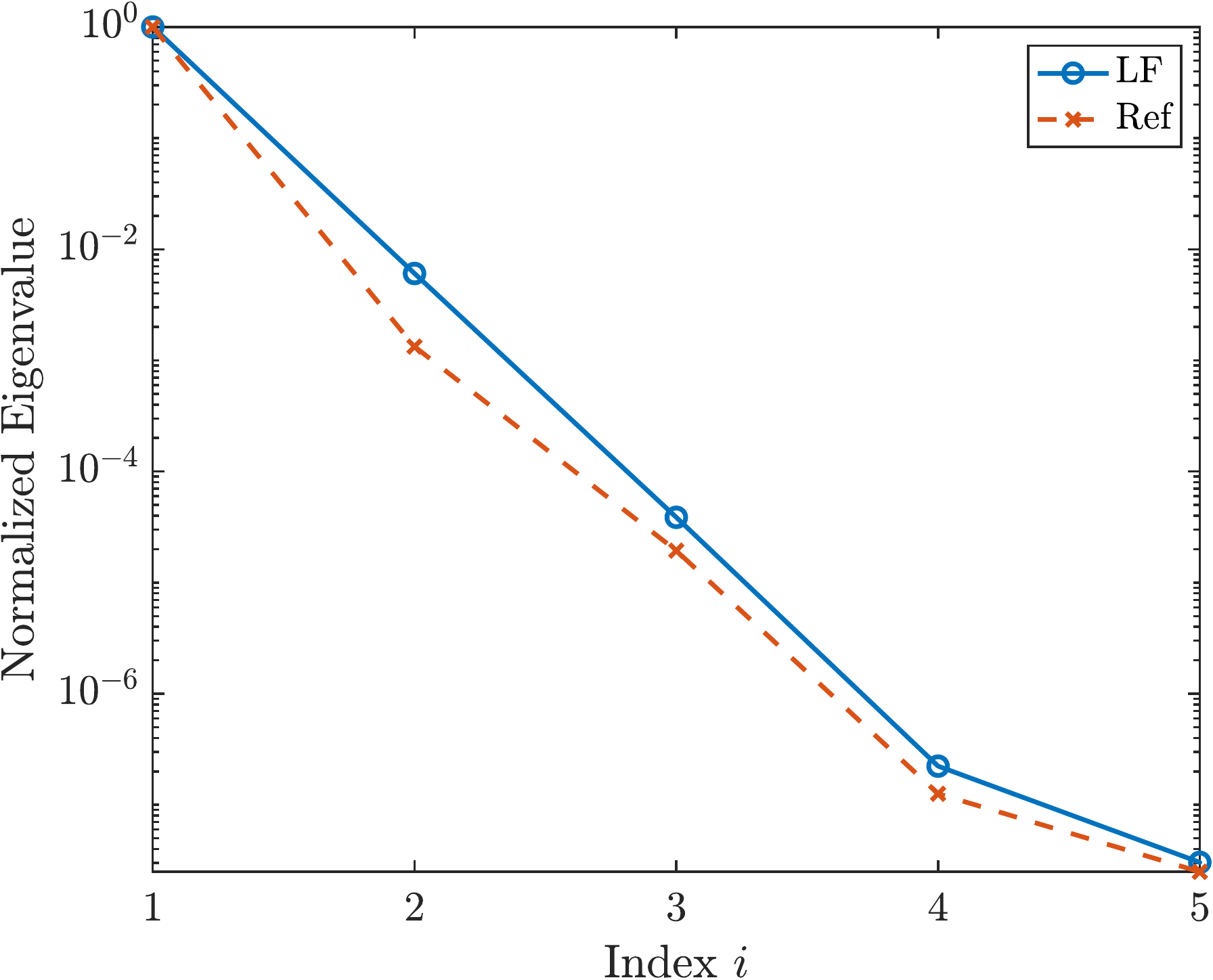}
\caption{Normalized eigenvalues of the LF and reference, Ref covariance matrices of the vertical velocity component along $y=0.5$ in Example I.}
\label{fig:LDC_Eig}	
\end{figure}

We now consider the performance of the error bound. In the following, we set $r=4 $ and $n = 15$. A useful metric in our assessment of bound performance is the error bound efficacy, calculated as the (average) ratio of the error bound and the true error{\color{black},} for different approximation ranks $r$ and number of HF samples $n$. An efficacy of $1$ indicates that the error bound is exact, and an efficacy of greater than $1$ implies the bound does not underestimate the error,which could happen as each of the bounds are only guaranteed to hold with some probability. Note that, while the bounded error is squared, we take the square root in our efficacy calculation. 
In Table \ref{tab:ldc_bound1}, we report the efficacy of the error bound  (\ref{eq:nu_bound_pract}) and associated probability (\ref{eq:prob_sum}), calculated with $t=2.0$. 
Note that the probability computed in (\ref{eq:prob_vector}) matches (\ref{eq:prob_sum}). The error is tightly bounded with high probability. 
  \begin{table}[htbp]
  \centering
      \caption{Practical error probability (\ref{eq:prob_sum}) and bound efficacy (\ref{eq:nu_bound_pract}) in Example I. Results are calculated as the average of $30$ repetitions.}
    \begin{tabular}{lllllll}
    \hline\noalign{\smallskip}
  QoI   & $r$ & $n$ & $\hat{n}$ & $N$ & Prob. (\ref{eq:prob_sum}) & Eff. (\ref{eq:nu_bound_pract})  \\
    \noalign{\smallskip}\hline\noalign{\smallskip}
 Vertical Velocity & $4$ & $15$ & $15$ & $200$ & $0.855$  & $1.83$\\
    \noalign{\smallskip}\hline
    \end{tabular}%
  \label{tab:ldc_bound1}%
\end{table}

In Figure \ref{fig:LDC_bound} (a) we present the pointwise reference error, calculated as the average of the entire HF ensemble for each point, and error bound (\ref{eq:nu_bound_i_pract}). Bound (\ref{eq:nu_bound_i_pract}) is very tight and closely matches the shape of the reference error over the entire interval of $x$.
Figure \ref{fig:LDC_bound} (b) depicts the efficacy for a given $(n,r)$ pair as the average of $30$ repetitions. For a very small number of HF samples the error bound is not effective, but as the number of samples increases above $n = 10$ we see the desired behavior of efficacy greater than $1$. In addition, the practical error bound does not exceed more than three times the BF error, implying evaluation of the bound in (\ref{eq:nu_bound_pract}) can be a useful tool for determining whether a given LF data-set is appropriate. Together, Figures \ref{fig:LDC_bound} (a) and (b) show that the error bound (\ref{eq:nu_bound_i_pract}) and (\ref{eq:nu_bound_pract}) are useful to assess pointwise and ensemble accuracy, respectively.
\begin{figure*}[ht!]
\centering
\begin{subfigure}{0.45\textwidth}
  \centering
  \includegraphics[width=\textwidth]{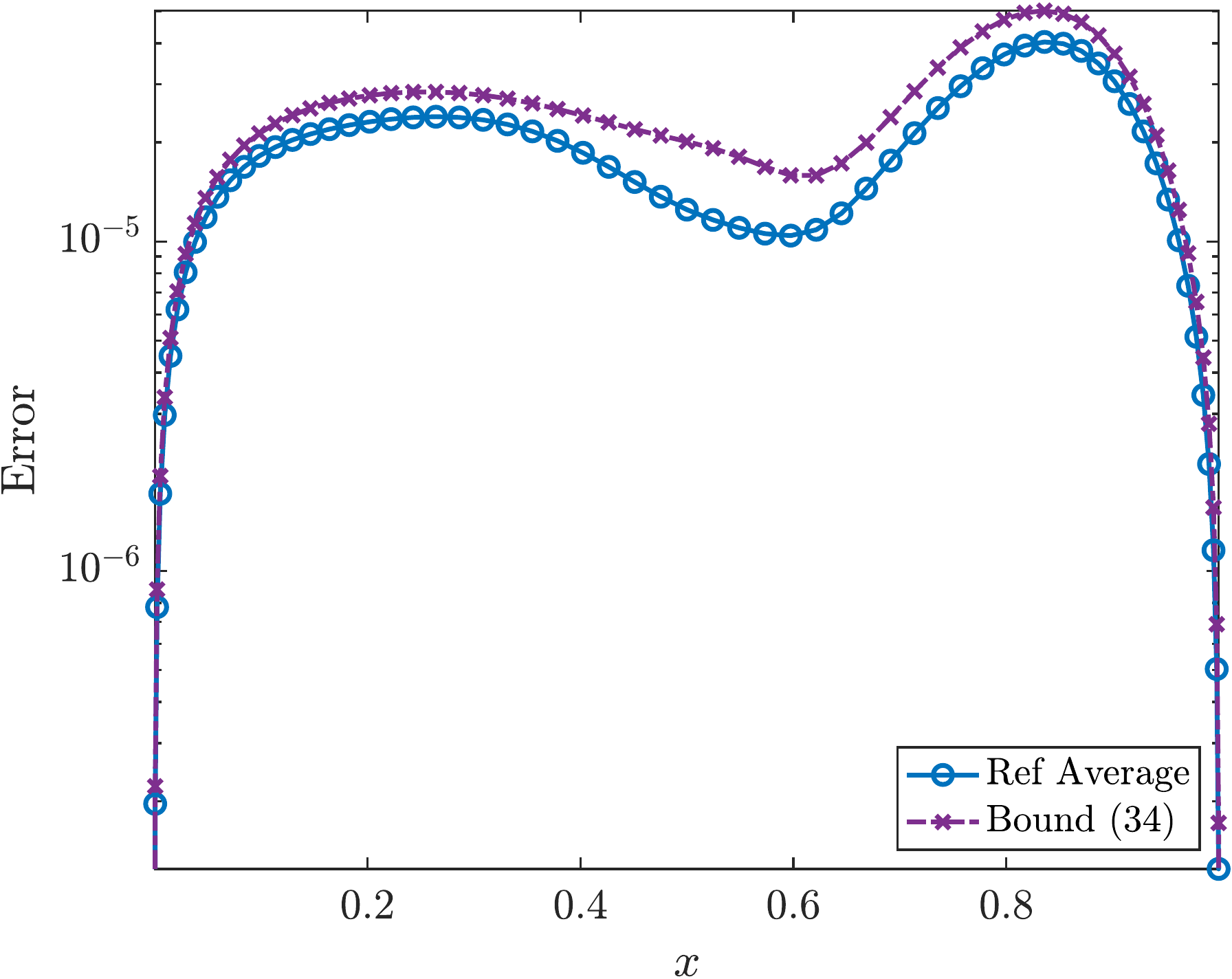}
  \caption{Pointwise true error and error bound (\ref{eq:nu_bound_i_pract}).  }
  \label{fig:LDC_point}
\end{subfigure}
\hfill
\begin{subfigure}{0.45\textwidth}
  \centering
  \includegraphics[width=\textwidth]{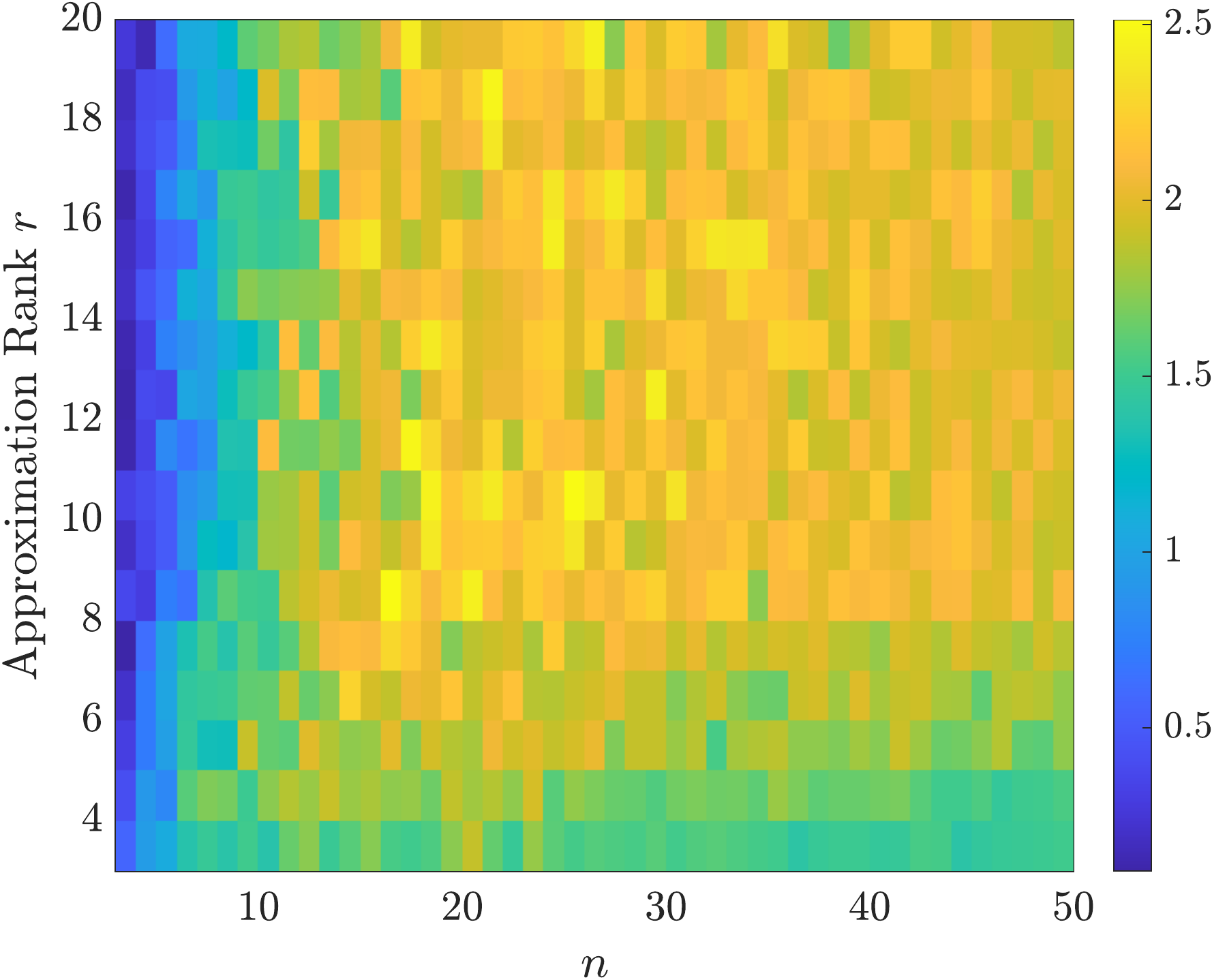}
  \caption{Practical error bound efficacy.}
  \label{fig:LDC_efficacy}
\end{subfigure}
 \caption[Error bound performance for the lid-driven cavity of Example I.]{Error bound performance for the lid-driven cavity of Example I. We present the average calculated from $30$ repetitions, where each repetition includes the estimate of $N=200$ samples. In (a) we use $n = 15$ HF samples and rank $r= 4$. In (b) the practical bound efficacy is calculated as the (average) ratio of the error estimated from the practical error bound (\ref{eq:nu_bound_pract}) and the true error for different approximation ranks $r$ and number of HF samples $n$. We use the same $n$ HF samples used to build the BF approximation to compute the bound.}
   \label{fig:LDC_bound}
\end{figure*}
\subsection{Example II: flow past a heated cylinder}
\label{sec:case_2_cylinder}

Gas turbines function at high combustor outlet temperatures and endeavor to mitigate thermal stress and fatigue through cooling flow passages. Cooling pin arrays are introduced to the turbine airfoil that enhance heat transfer though both additional surface area and increased turbulence \cite{lyall2011heat}. Extracted coolant flow to cool the turbine can comprise as much as $10 \%$ of the engine's core flow making it an important factor in turbine performance \cite{hill1992mechanics}.  

We approximate the cooling of a turbine blade with a two-dimensional incompressible flow past an array of pins, taking inspiration from \cite{constantine2009hybrid}. The domain is a rectangle with corners $(-0.2,-0.1)$ and $(1.0,0.1)$ and a circle centered at $(0,0)$ with radius $0.05$ as seen in Figure \ref{fig:GT_mesh}. The dimensions are in consistent units. The vertical direction is periodic to simulate a pin separation of $L/D = 1$, $D$ being the diameter of the cylinder and $L$ the distance between neighboring pins. Inflow and outflow are applied from left to right and the cylinder is no-slip. We assume the flow to be fully turbulent with $Re_D = 10^6$, the Reynolds number measured with the characteristic length as the cylinder diameter. 
\begin{figure*}[ht!]
\centering
\includegraphics[width=\textwidth]{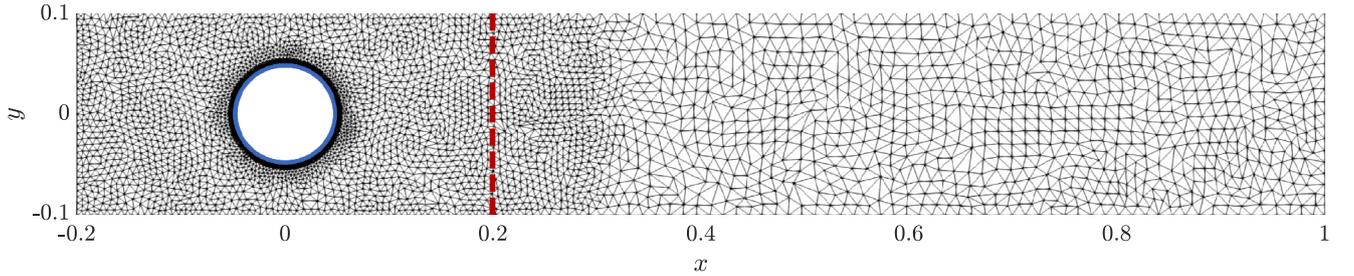}
\caption[Cylinder geometry and coarse computational mesh of Example II.]{Cylinder geometry and coarse computational mesh of Example II. The blue solid circle indicates $T_{\text{cylinder}}$, the temperature of the cylinder surface and the red dashed line indicates $T_{x=0.2}$, the temperature along a vertical line at $x=0.2$.}
\label{fig:GT_mesh}	
\end{figure*}

In practice, the fluid flow arriving to the cylinders is a product of numerous upstream component interactions such as the turbulators and slots \cite{constantine2009hybrid}. The inlet velocity profile is modeled as the sum of two waveforms with different frequencies and random amplitudes as 
\begin{align}
V \vert_{\text{inlet}} (y, Y_1, Y_2) =& 156.8 (1+ h_1 Y_1 \cos( 20 \pi y) \\
&+ h_2 Y_2\cos( 100 \pi y) ),
\label{eq:inlet_vel}
\end{align}
where $y \in [-0.1, 0.1]$. The inlet temperature is modeled as a Gaussian pulse of random mean and constant amplitude given by
\begin{equation}
T \vert_{\text{inlet}}(y, Y_3) = 300 + 100 \: \exp  \left( -\frac{(y - 0.05 Y_3)^2 }{2 \times \: 0.01^2} \right),
\label{eq:inlet_temp}
\end{equation}
where the mean is constrained within  $\pm 0.05$. The heat flux over the cylinder wall is defined by an exponential as
\begin{equation}
\left.\frac{\partial T}{\partial n} \right\vert_\text{cyl} ( \theta, Y_4) = -50 \: \exp \left(-\left(0.1+ \frac{h_3 Y_4 \cos \theta}{2} \right) \right),
 \label{eq:cylinder_flux}
\end{equation}
where $\theta \in [0, 2\pi]$ is the angle from the leading edge. Four stochastic inputs, $Y_1$, $Y_2$, $Y_3$, and $Y_4$, all with uniform distribution $ U[-1, 1]$, alongside three constants, $h_1 = 0.7$, $h_2 = 0.2$ and $h_3 = 0.9$, determine the boundary conditions.

The problem is solved using the Reynolds averaged Nav\hyp{}ier-Stokes (RANS) equations in \texttt{PHASTA}, a parallel hierarchic adaptive stabilized transient analysis computational fluid dynamics (CFD) code \cite{whiting2001stabilized}. 
The turbulence closure model is Spalart-Allmaras \cite{spalart1992one, pope2001turbulent}. The fluid is modeled as air with viscosity $\nu = 1.568 \times 10^{-5}$ $[\text{m}^2\text{s}^{-1}]$, density $\rho = 1.177$ $[\text{kg} ~ \text{m}^{-3}] $, scalar diffusivity $23.07 \times 10^{-6}$ $[\text{m}^2 \text{s}^{-1}] $, thermal conductivity $k = 26.62 \times 10^{-3}$ $[\text{W} \text{m}^{-1} \text{K}^{-1}]${\color{black},} and turbulent Prandtl number $0.7$. 

A coarse $40{,}000$ element mesh and a fine $110{,}000$ element mesh are employed as the LF and HF models, respectively. The HF model ensures $y^+ <1$ is maintained on the cylinder surface and a steady state solution is attained. The computational cost of the HF model is approximately $50$ times greater than the LF model. Two QoIs are considered, the temperature of the cylinder surface, $T_{\text{cylinder}}$ and the temperature along a vertical line through the geometry at $x = 0.2$, $T_{x=0.2}$; see Figure \ref{fig:GT_mesh}. 

\subsubsection{Results}
\label{sec:cylinder_results}

We consider BF estimates that use $N=200$ LF and $n$ HF samples. The average relative error for the mean and variance of the temperature on the cylinder surface, with a PC expansion of order $p=6$, is plotted in Figure \ref{fig:GT_BR_Cylinder}. For $n = 10$ the LF model out-performs both HF and BF methods in estimating the mean and variance. As $n$ increases the relative errors in the HF and BF methods decrease. Further, it appears that a truncated basis of $r=8$ is sufficient to capture the solution behavior. This is a notable reduction from the full PC expansion of $P = 210$ terms. The BF relative error in variance does not descend below $1 \%$. This result is acceptable given the reference solution error in variance estimation is approximately $1 \%$, and can be considered as the limit of the BF performance. 
\begin{figure*}[ht!]
\centering
\includegraphics[width=\textwidth]{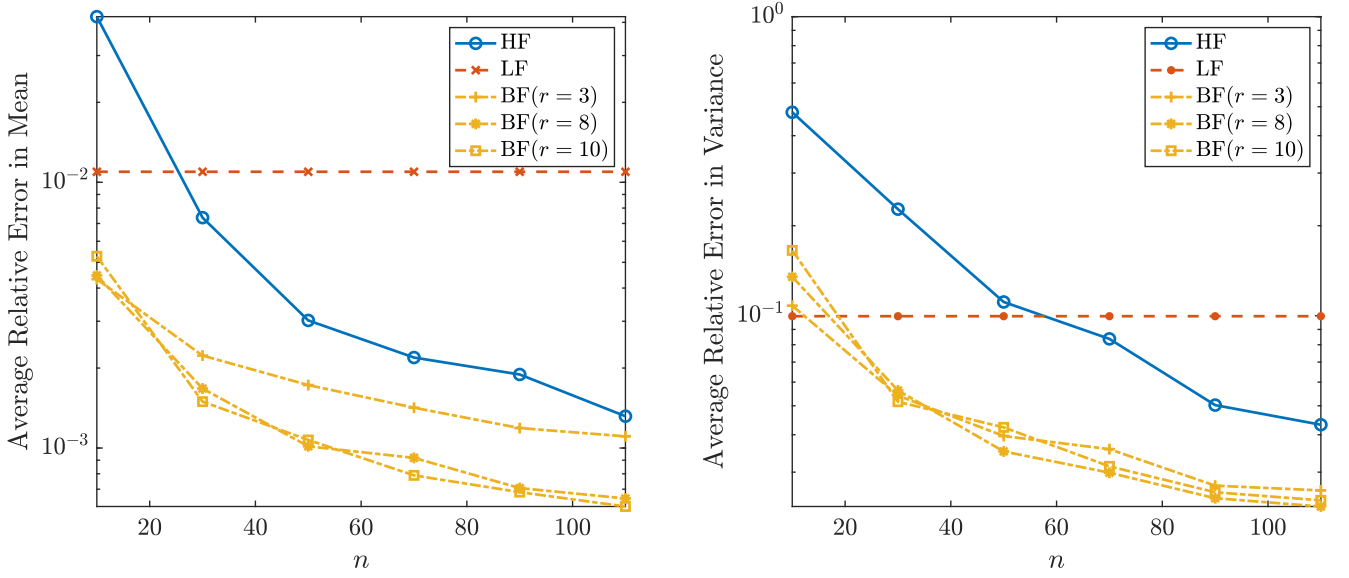}
\caption[The average relative error for the (left) mean, and (right) variance of the temperature on the cylinder surface in Example II.]{The average relative error for the (left) mean, and (right) variance of the temperature on the cylinder surface in Example II. Plotted are the HF, LF, and BF estimates for approximation rank $r$. The average is calculated from $100$ repetitions.}
\label{fig:GT_BR_Cylinder}	
\end{figure*}

Figure \ref{fig:GT_RB_Mid} depicts the average relative error for the mean and variance of the temperature along a vertical line at $x = 0.2$, with a PC expansion of order $p = 6$. Given that the LF errors for both the mean and variance are small, $<0.1$, the utility of the BF method is reduced. In essence, if a LF estimate provides satisfactory results then there is little to be gained through HF samples and a BF approximation. As the number of HF samples increases, we observe crossover points where the the HF estimate is better than the BF estimate, which is typical of multi-fidelity approximations \cite{de2020transfer}. There is also a notable improvement in the BF estimates when the approximation rank is raised from $r=3$ to $8$. 
\begin{figure*}[ht!]
\centering
\includegraphics[width=\textwidth]{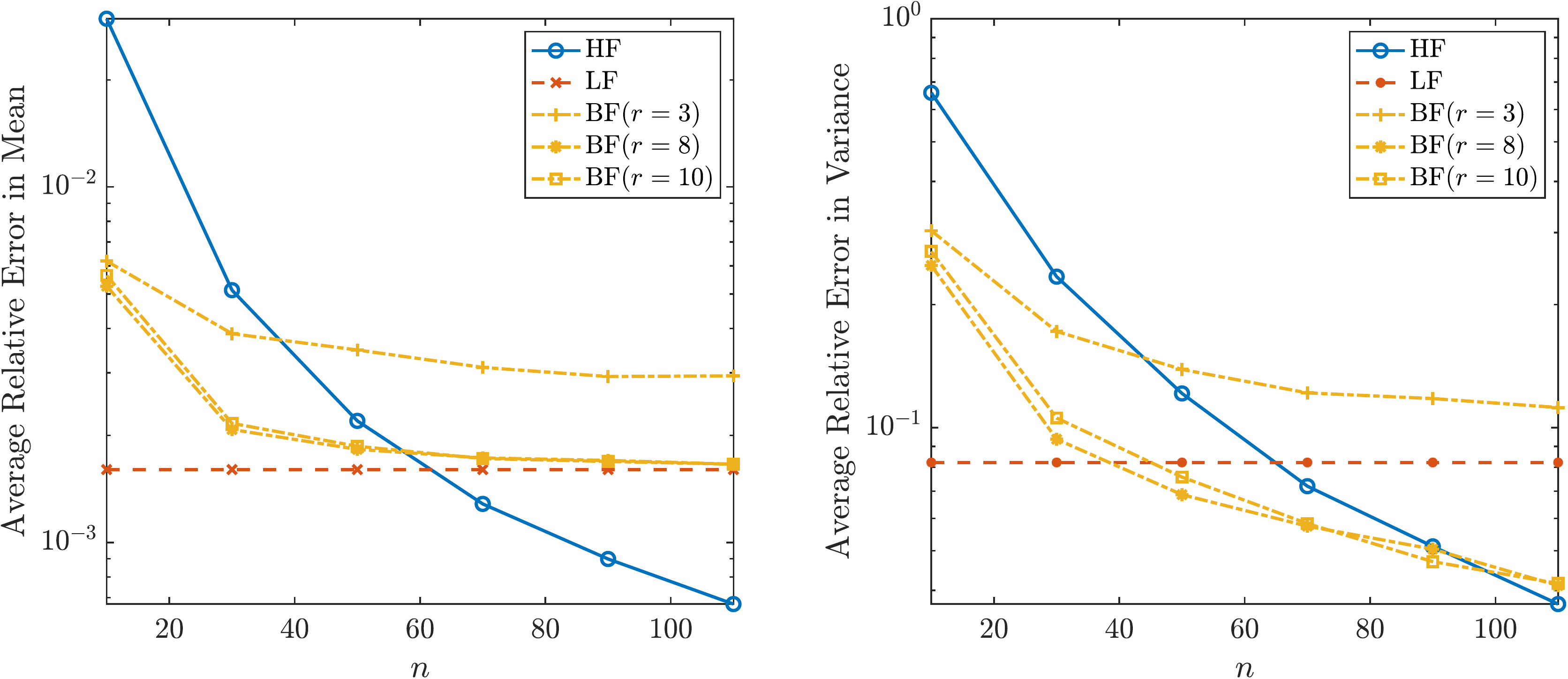}
\caption[The average relative error for the (left) mean, and (right) variance of the temperature along a vertical line at $x = 0.2$ in Example II.]{The average relative error for the (left) mean, and (right) variance of the temperature along a vertical line at $x = 0.2$ in Example II. Plotted are the HF, LF, and BF estimates for approximation rank $r$. The average is calculated from $100$ repetitions.}
\label{fig:GT_RB_Mid}	
\end{figure*}
For both QoIs, the variance is challenging to estimate. The temperature field of four different realizations is plotted in Figure \ref{fig:GT_field}. It is apparent that, depending on the stochastic inputs, the temperature field about the cylinder has significant variance. 
   \begin{figure*}[!htb]
     \centering
          \begin{minipage}[b]{0.4\textwidth}
       \begin{subfigure}[b]{\linewidth}
	   \includegraphics[width=\textwidth]{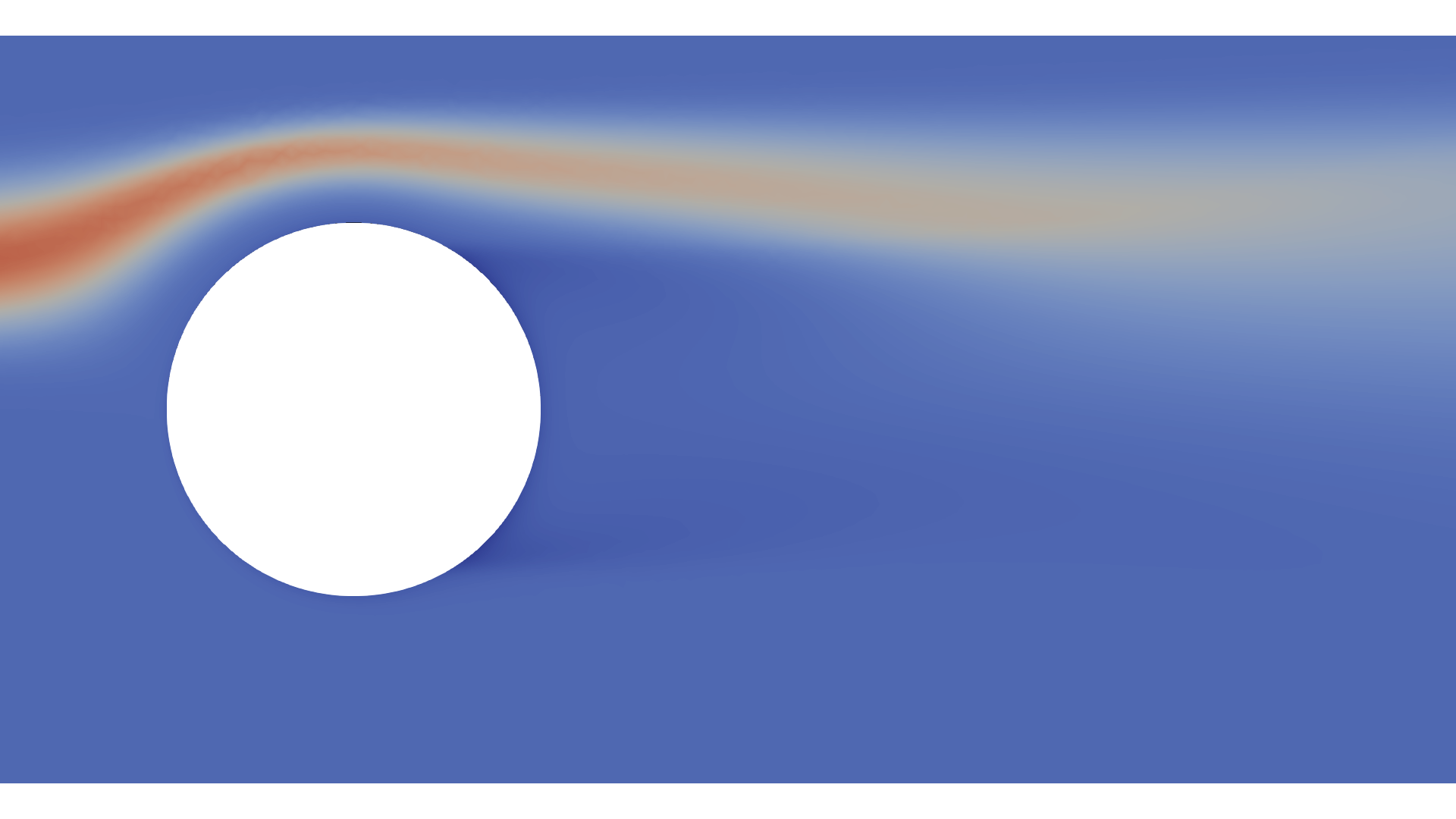}
       \end{subfigure}\\[\baselineskip]
       \begin{subfigure}[b]{\linewidth}
       \vspace{-0.6cm}
	    \includegraphics[width=\textwidth]{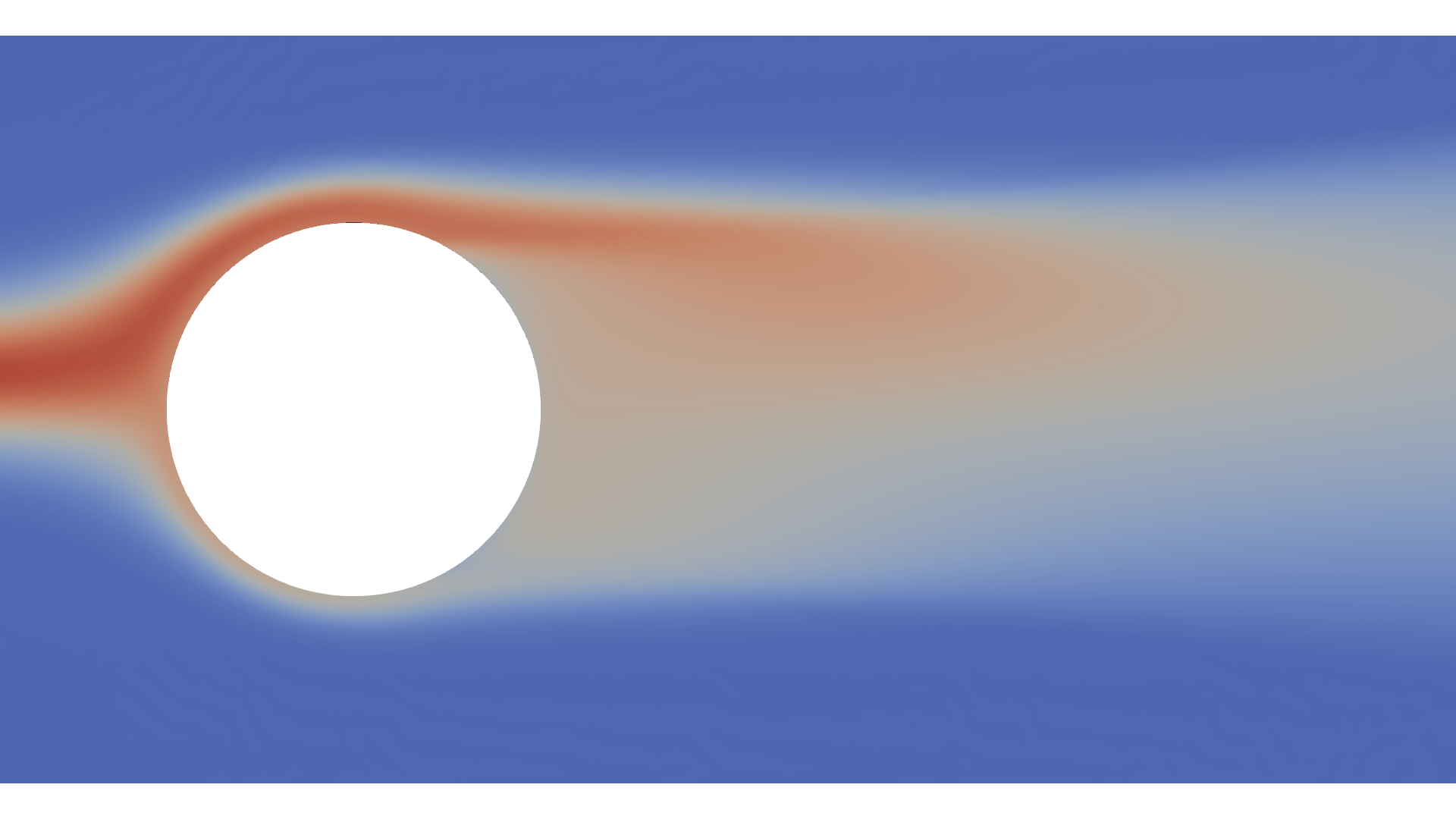}
       \end{subfigure}
     \end{minipage}
      \hfill
        \begin{minipage}[b]{0.4\textwidth}
       \begin{subfigure}[b]{\linewidth}
	    \includegraphics[width=\textwidth]{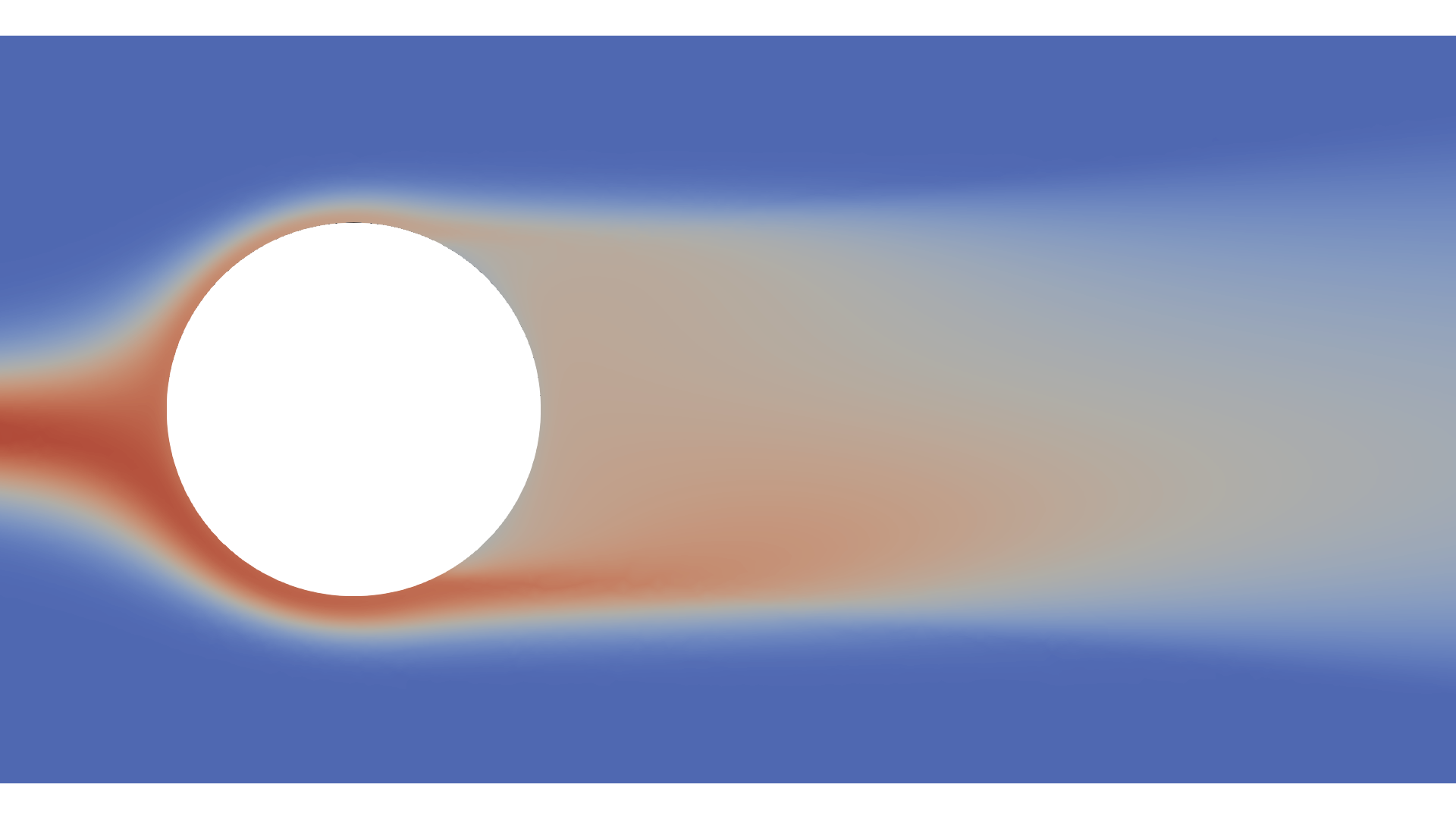}
       \end{subfigure}\\[\baselineskip]
       \begin{subfigure}[b]{\linewidth}
       \vspace{-0.6cm}
	    \includegraphics[width=\textwidth]{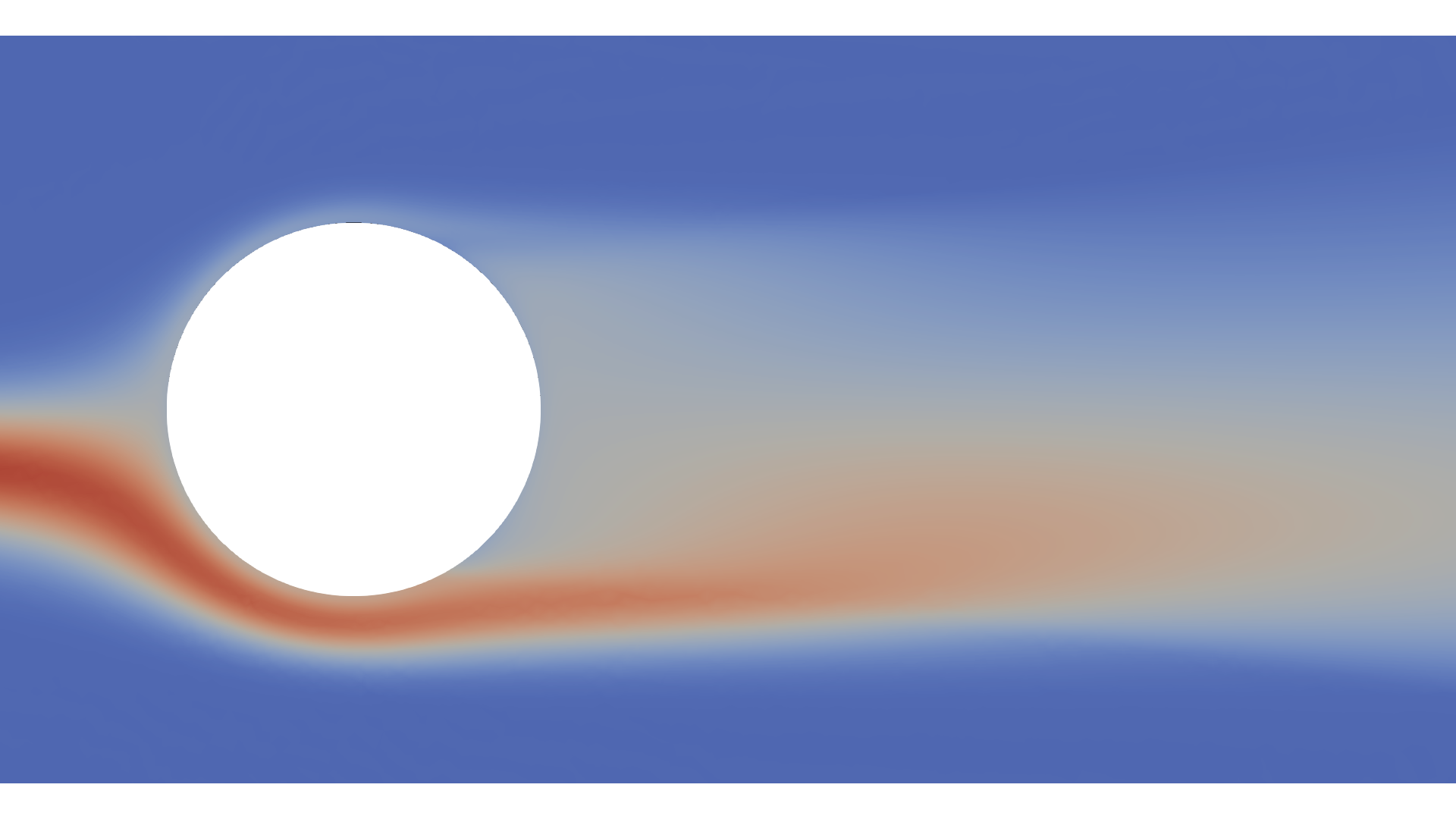}
       \end{subfigure}
     \end{minipage}
     \hfill
     \begin{subfigure}[b]{0.1\textwidth}
	  \includegraphics[width=\textwidth]{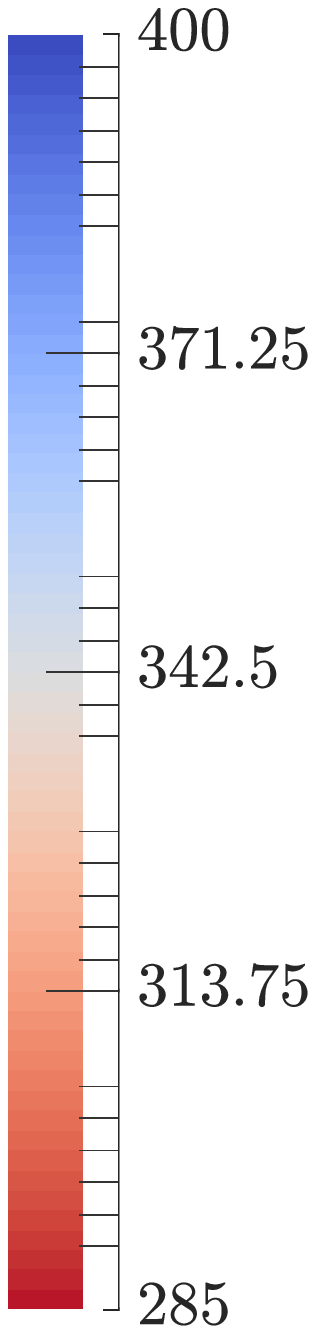}
     \end{subfigure}
    
     \caption{Four realizations of the temperature field about the cylinder in Example II.}
     \label{fig:GT_field}
   \end{figure*}

Setting $n=50$ HF samples and an approximation rank of $r=8$, we compare the mean and variance estimates in Figures \ref{fig:GT_cy_mean_var} and \ref{fig:GT_mid_mean_var}. It is evident in Figure \ref{fig:GT_cy_mean_var} that the BF estimate of the mean and variance of the cylinder surface temperature are superior to the LF and HF, with $n$ samples. In contrast, in Figure \ref{fig:GT_mid_mean_var}, the LF and HF estimates of the variance of temperature along the vertical line at $x=0.2$ are comparable to the BF. One aspect in which the LF model is inferior to the HF is the lack of resolution around the cylinder surface. This may explain the greater utility of the BF approach in estimating cylinder surface temperature as opposed to the vertical line at $x=0.2$.
\begin{figure*}[ht!]
\centering
\includegraphics[width=\textwidth]{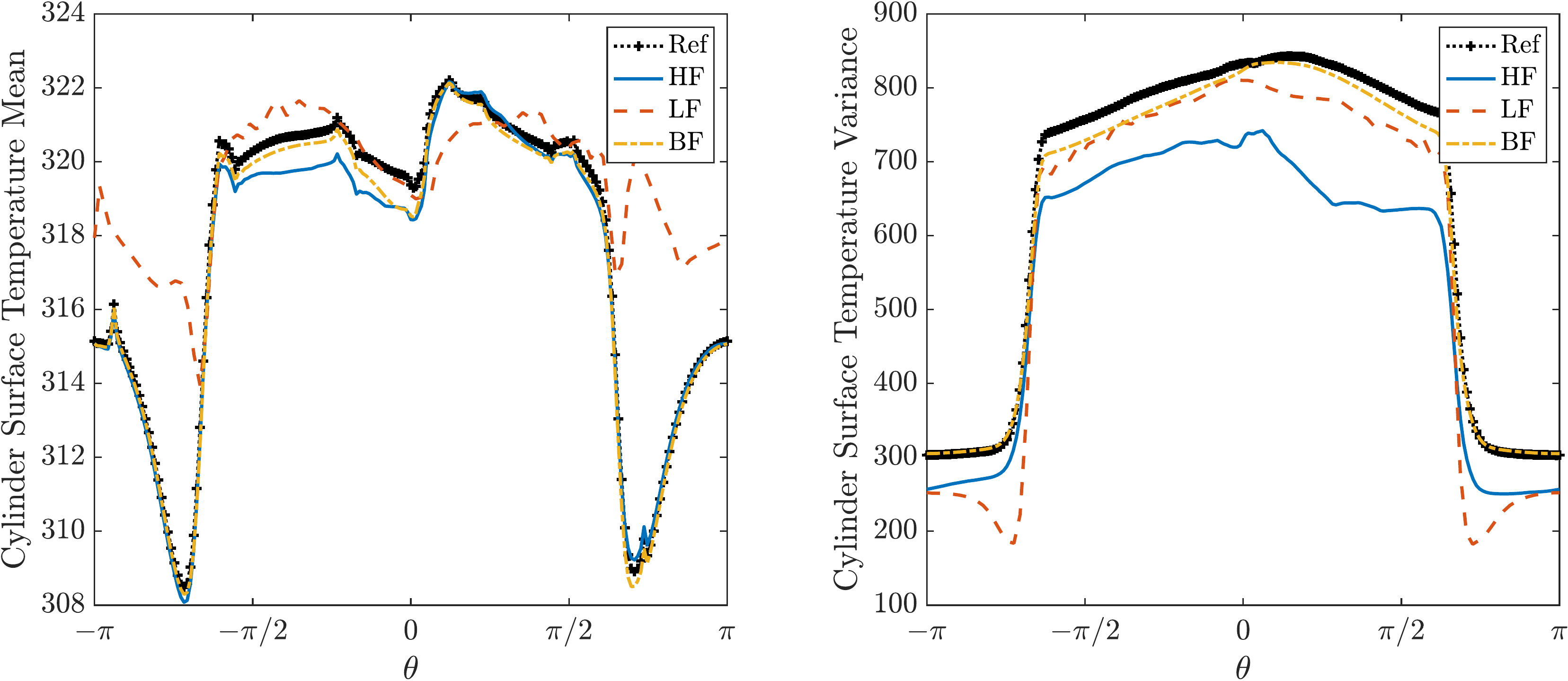}
\caption[Cylinder surface temperature in Example II.]{Cylinder surface temperature in Example II. Shown are (left) mean and (right) variance for rank $r = 8$ with $n=50$ HF samples. The reference solution is given by Ref alongside the HF, LF and BF estimates.}
\label{fig:GT_cy_mean_var}	
\end{figure*}
\begin{figure*}[ht!]
\centering
\includegraphics[width=\textwidth]{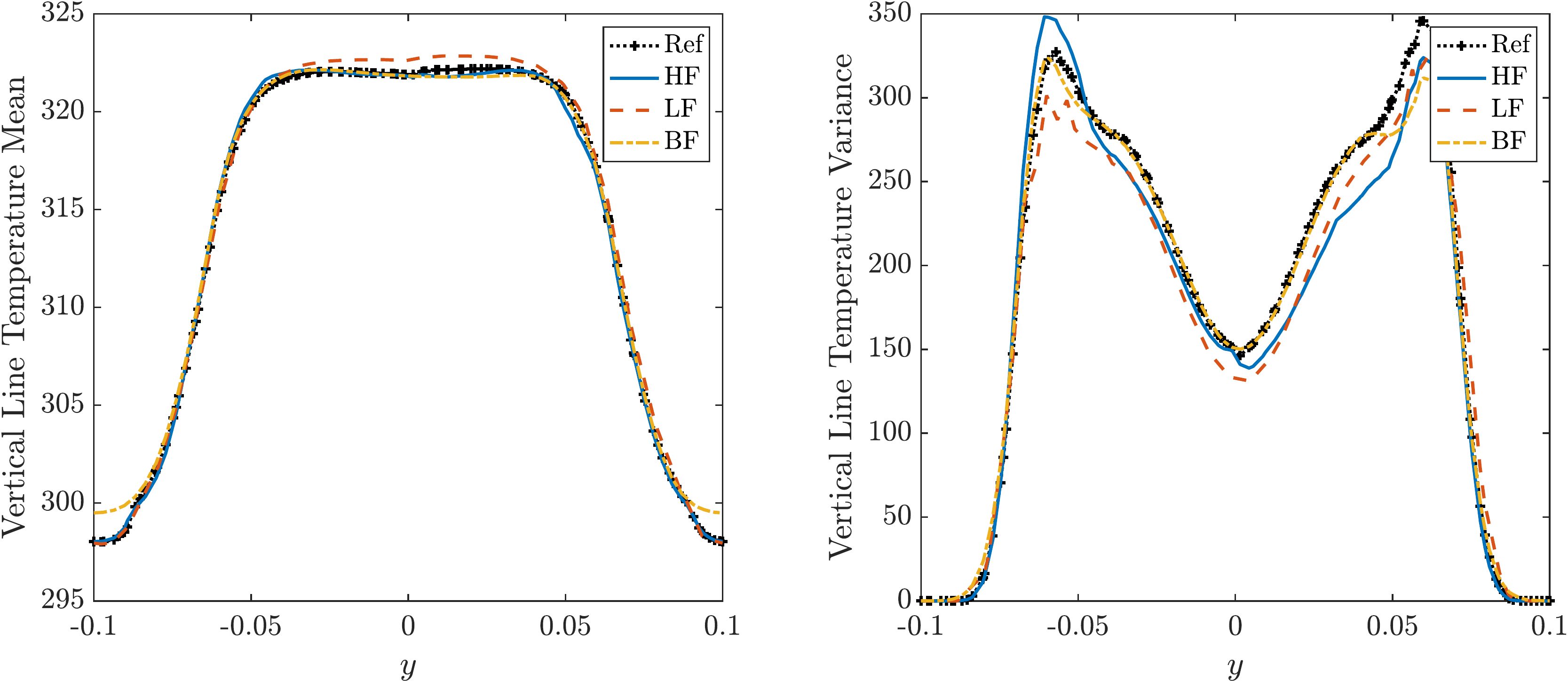}
\caption[Temperature along a vertical line at $x=0.2$ in Example II.]{Temperature along a vertical line at $x=0.2$ in Example II. Shown are (left) mean and (right) variance for rank $r = 8$ with $n=50$ HF samples. The reference solution is given by Ref alongside the HF, LF and BF estimates.}
\label{fig:GT_mid_mean_var}	
\end{figure*}

The normalized eigenvalues of the LF and reference covariance matrices for both QoIs are shown in Figure \ref{fig:GT_Eig}. Both QoIs demonstrate rapid decay in their eigenvalues. The eigenvalues of the LF and reference solution align well for $T_{\text{cylinder}}$ while in $T_{x=0.2}$ a disparity develops for indices greater than $3$. This observation agrees with the results of Figures \ref{fig:GT_BR_Cylinder} and \ref{fig:GT_RB_Mid} in that the reduced basis estimate arising from $T_{\text{cylinder}}$, in which the LF eigenvalue decay corresponds well to its reference solution, provides an accurate estimate of the mean and variance. 
\begin{figure}[ht!]
\centering
\includegraphics[width=0.5\textwidth]{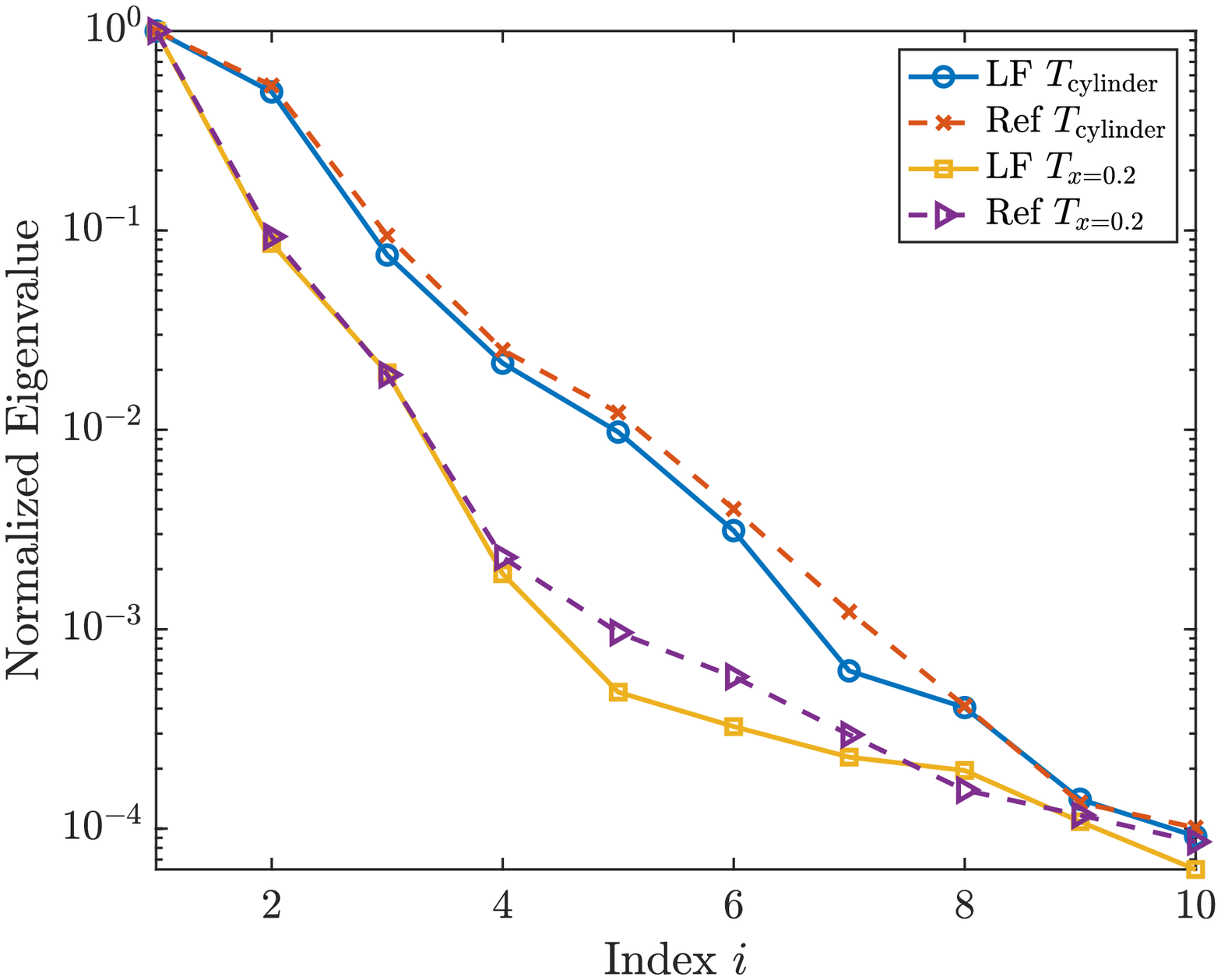}
\caption{Flow past cylinder, normalized eigenvalues of the LF and reference, Ref covariance matrices for both $T_{\text{cylinder}}$ and $T_{x=0.2}$ for Example II.}
\label{fig:GT_Eig}	
\end{figure}

We next look at the performance of the error bound. We set $r=8$, $n = 50$, and $t=2.0$. 
In Table \ref{tab:gt_bound1}, we present the error bound efficacy (\ref{eq:nu_bound_pract})  and associated probability (\ref{eq:prob_sum}). 
We observe that error is tightly bounded with high probability. The bound probabilities for both QoI are the same, and we also note that these matched the pointwise probability computed in (\ref{eq:prob_vector}).  
  \begin{table}[htbp]
      \caption{Practical error probability (\ref{eq:prob_sum}) and bound efficacy (\ref{eq:nu_bound_pract}) in Example II. Results are calculated as the average of $30$ repetitions.}
  \centering
    \begin{tabular}{ccccccc}
    \hline\noalign{\smallskip}
  QoI   & $r$ & $n$ & $\hat{n}$ & $N$ & Prob. (\ref{eq:prob_sum}) & Eff. (\ref{eq:nu_bound_pract})  \\
    \noalign{\smallskip}\hline\noalign{\smallskip}
 Cylinder Surface  & $8$ & $50$ & $50$ & $200$ & $0.910$  & $1.48$\\
 Vertical Line  & $8$ & $50$ & $50$ & $200$ & $0.910$ & $1.65$ \\
    \noalign{\smallskip}\hline
    \end{tabular}%

  \label{tab:gt_bound1}%
\end{table}

In Figures \ref{fig:GT_cylinder} (a) and \ref{fig:GT_mid} (a),  we present the pointwise reference average and error bound (\ref{eq:nu_bound_i_pract}) for the cylinder surface and vertical line, respectively. For both QoI the practical bound follows the shape of the reference error and is within a factor of two. 
 
Figures \ref{fig:GT_cylinder} (b) and \ref{fig:GT_mid} (b) depict the efficacy for a given $(n,r)$ pair as the average of $30$ repetitions for cylinder surface and vertical line. Where few HF samples are available we see the bound does not perform well, but as the number of HF samples increases the efficacy is greater than $1$, as necessary. We also observe that the efficacy is tighter for smaller values of $r$. The efficacy of the practical error bound for both QoI does not exceed more than three, demonstrating the utility of (\ref{eq:nu_bound_pract}) to determine the eligibility of a LF data-set for BF estimation.
\begin{figure*}[ht!]
\centering
\begin{subfigure}[t]{0.45\textwidth}
  \centering
  \includegraphics[width=\textwidth]{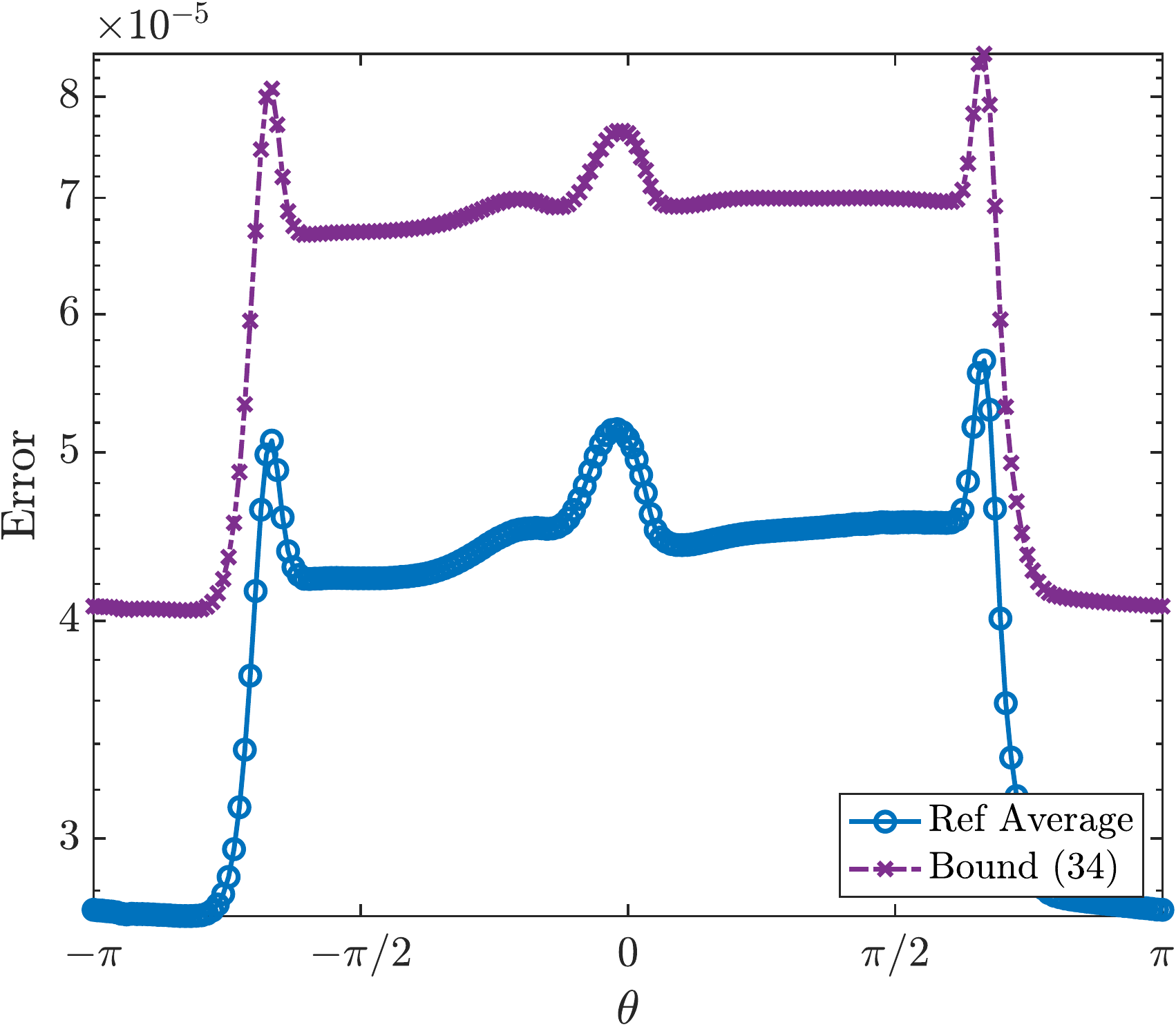}
  \caption{Pointwise true error and error bound (\ref{eq:nu_bound_i_pract}).  }
  \label{fig:GT_cylinder_bound}
\end{subfigure}
\hfill
\begin{subfigure}[t]{0.45\textwidth}
  \centering
  \includegraphics[width=\textwidth]{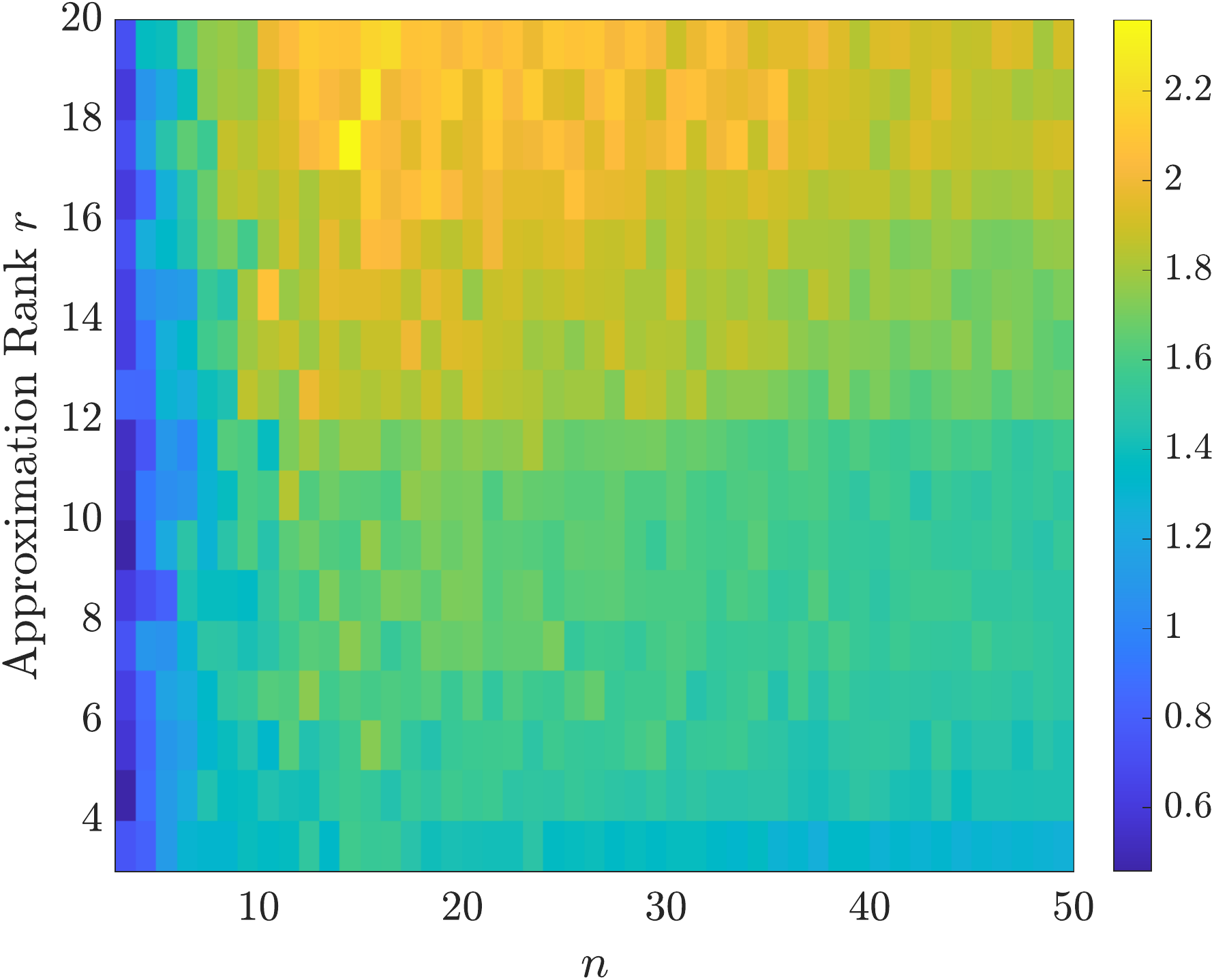}
  \caption{Practical error bound efficacy.}
  \label{fig:GT_cylinder_efficacy}
\end{subfigure}
 \caption[Practical bound performance for the cylinder surface temperature of Example II.]{Practical bound performance for the cylinder surface temperature of Example II. We present the average calculated from $30$ repetitions, where each repetition includes the estimate of $N=100$ samples. In (a) we use $n = 20$ HF samples and rank $r= 8$. In (b) the efficacy is calculated as the (average) ratio of the error estimated from the practical error bound (\ref{eq:nu_bound_pract}) and the true error for different approximation ranks $r$ and number of HF samples $n$.  We use the same $n$ HF samples used to build the BF approximation to compute the bound.}
\label{fig:GT_cylinder}
\end{figure*}
\begin{figure*}[ht!]
\begin{subfigure}[t]{0.45\textwidth}
  \centering
  \includegraphics[width=\textwidth]{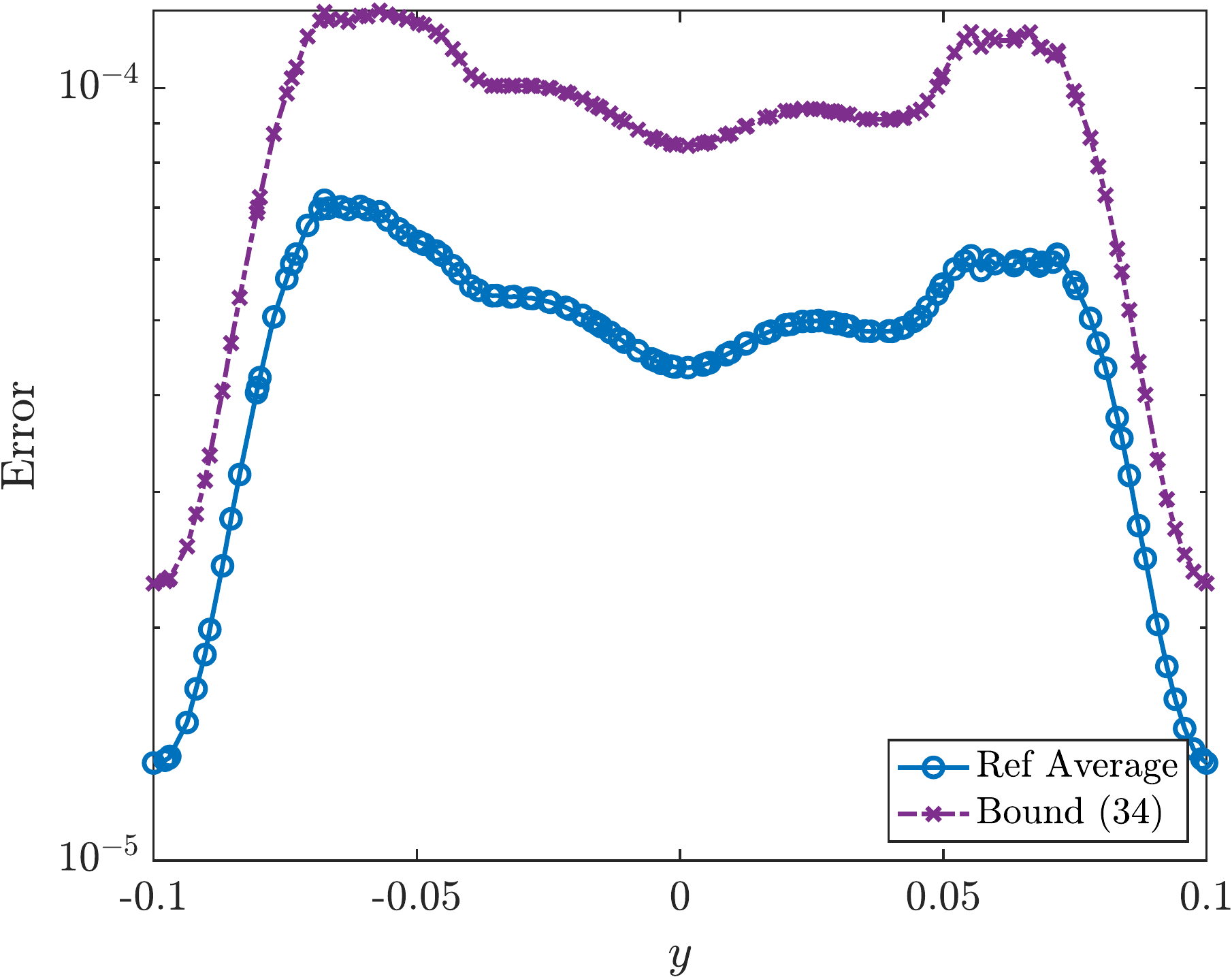}
  \caption{Pointwise true error and error bound (\ref{eq:nu_bound_i_pract}).  }
  \label{fig:GT_mid_bound}
\end{subfigure}
\hfill
\begin{subfigure}[t]{0.45\textwidth}
  \centering
  \includegraphics[width=\textwidth]{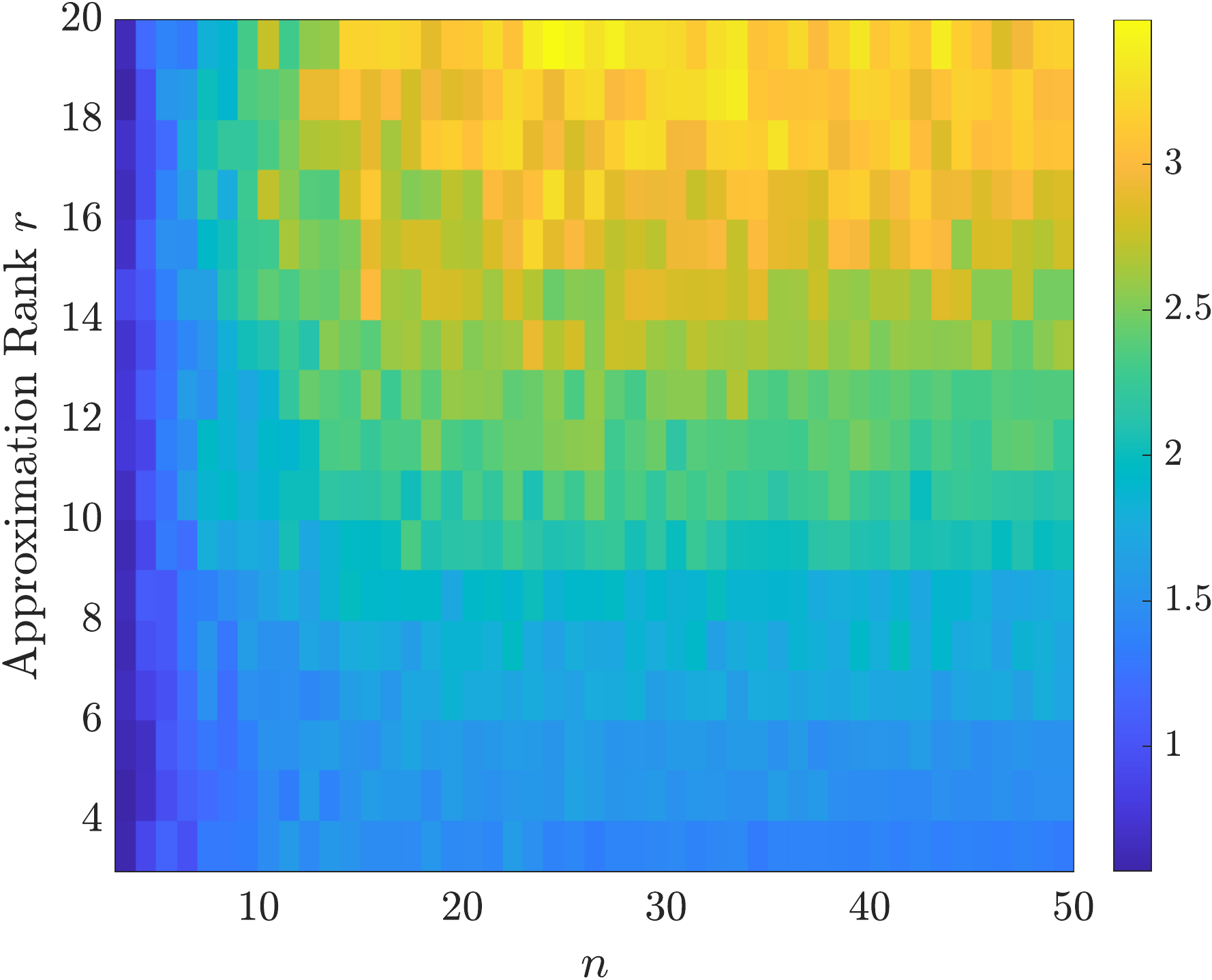}
  \caption{Practical error bound efficacy.}
  \label{fig:GT_mid_efficacy}
\end{subfigure}
 \caption[Error bound performance for the vertical line temperature of Example II.]{Error bound performance for the vertical line temperature of Example II. We present the average calculated from $30$ repetitions, where each repetition includes the estimate of $N=100$ samples. In (a) we use $n = 50$ HF samples and rank $r= 8$. In (b) the practical bound efficacy is calculated as the (average) ratio of the error estimated from the practical error bound (\ref{eq:nu_bound_pract}) and the true error for different approximation ranks $r$ and number of HF samples $n$.  We use the same $n$ HF samples used to build the BF approximation to compute the bound.}
 \label{fig:GT_mid}
\end{figure*}

\subsection{Example III: flow around NACA airfoil} 
\label{sec:case_3_airfoil}

The final example we consider is a NACA $4412$ airfoil solved with a Reynolds number of $1.52 \times 10^6$ at a low angle-of-attack (AoA) as presented in \cite{skinner2019reduced}. An airfoil of chord length $1.0 \: \text{m}$ {\color{black}is} modeled in a computational domain of length $999 \: \text{m}$ in the stream-wise ($x$-direction), vertical height $998 \: \text{m}$ ($y$-direction), and span-wise width $2 \: \text{m}$ ($z$-direction). The domain and accompanying boundary conditions are specified in Figure \ref{fig:Airfoil_schematic_pc} following the work of \cite{diskin2015grid}. The red line on the boundary indicates inflow, while the blue indicates outflow. The $\pm y$ and $\pm z$ boundaries are set as inviscid and impenetrable to model a two-dimensional problem. The airfoil surface is subject to a no-slip condition. We examine variations in the model geometry, varying the maximum camber $m$, the position of maximum camber $p$ and the maximum thickness $t$ from their nominal 4412 airfoil values. We also vary the angle of attack (AoA). All parameters are modeled by uniform random variables as summarized in Table \ref{tab:airfoil_pc}.
\begin{figure*}[ht!]
\centering
\includegraphics[width=\textwidth]{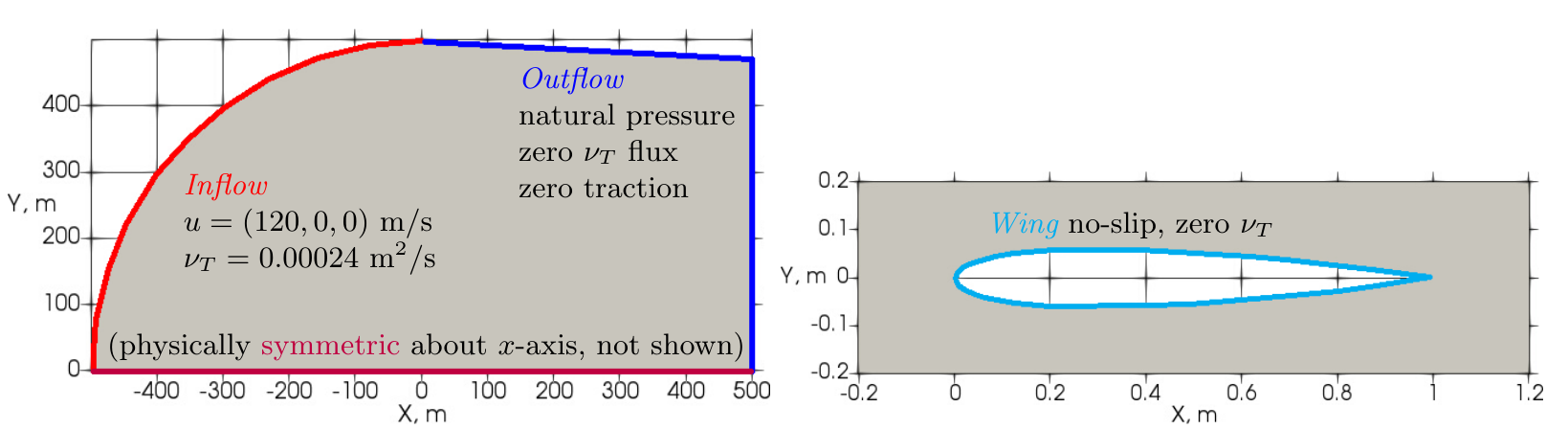}
\caption[Schematic of the computational domain, initial NACA 0012 airfoil and accompanying boundary conditions used in Example III.]{Schematic of the computational domain, initial NACA 0012 airfoil and accompanying boundary conditions used in Example III. The geometry is deformed to map into a NACA series airfoil, with 4412 defining the nominal geometry. The inflow velocity is given by $u$ and $\nu_T$ is the kinematic turbulence viscosity. Figure is adapted from \cite{skinner2019reduced}.}
\label{fig:Airfoil_schematic_pc}	
\end{figure*}
  \begin{table}[htbp]
  \centering
      \caption{Uncertain input variables for Example III. }
    \begin{tabular}{cccc}
    \hline\noalign{\smallskip}
     Parameter  & Symbol  & Distribution\\
    \noalign{\smallskip}\hline\noalign{\smallskip}
\text{Maximum camber}  & $m$  & $U [0.032, 0.048]$ \\
\text{Location of max camber}  & $p$ & $U [0.32, 0.48]$ \\
\text{Thickness}  & $t$ &  $U [0.096, 0.144]$\\
\text{Angle of Attack}  & $\alpha$  & $U [0^{\circ}, 6^{\circ}]$  \\
    \noalign{\smallskip}\hline
    \end{tabular}%
  \label{tab:airfoil_pc}%
\end{table}

The QoI considered is the coefficient of pressure, $C_p$, on the surface of the airfoil calculated as
\begin{equation}
C_p = \frac{p - p_\infty}{\frac{1}{2} \rho V_\infty^2},
\end{equation}
where $p$ is the pressure at specific point on the airfoil surface, while $p_\infty$, $\rho_\infty$, and $V_\infty$ are the pressure, density and velocity of the free stream flow. To mitigate bias in the QoI from a concentration of points at the leading and trailing edges, we interpolate $C_p$ onto $200$ evenly spaced points a\hyp{}round the airfoil surface prior to applying the BF method. 

As in Example II, the simulation is solved using the RANS equations in \texttt{PHASTA} \cite{whiting2001stabilized}, with the Spalart-Allmaras (SA) turbulence closure model \cite{spalart1992one, pope2001turbulent}. The LF and HF meshes are depicted in Figures \ref{fig:Airfoil_mesh_pc} (a) and (b), respectively. The LF mesh consists of $9{,}000$ elements while the HF mesh is refined quadratically from near the airfoil to the simulation domain boundaries and has $241{,}000$ elements. In contrast to the LF model, the HF model adequately resolves the shear layer downstream of the airfoil and wall effects. The computational expense of the LF model, is however, $498$ times smaller than its HF counterpart. For further detail on the airfoil model implementation please refer to \cite{skinner2019reduced}. 
\begin{figure*}[ht!]
\centering
\begin{subfigure}{0.45\textwidth}
  \centering
  \includegraphics[width=\textwidth]{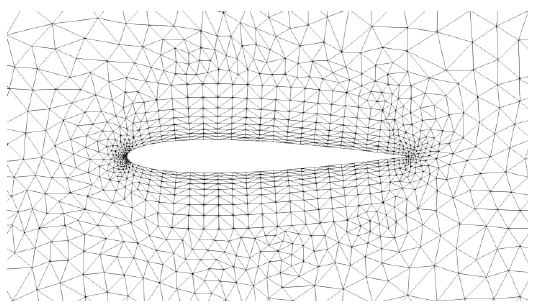}
  \caption{LF mesh.}
  \label{fig:Airfoil_meshL_pc}
\end{subfigure}
\begin{subfigure}{0.45\textwidth}
  \centering
  \includegraphics[width=\textwidth]{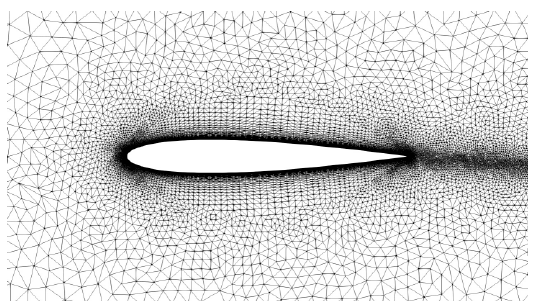}
  \caption{HF mesh.}
  \label{fig:Airfoil_meshH_pc}
\end{subfigure}
\caption[Computational mesh (a) LF model and (b) HF model for the NACA 0012 airfoil used in Example III.]{Computational mesh (a) LF model and (b) HF model for the NACA 0012 airfoil used in Example III. The geometry is deformed to map into a NACA series airfoil. Figure is adapted from \cite{skinner2019reduced}.} 
\label{fig:Airfoil_mesh_pc}
\end{figure*}
\subsubsection{Results}
\label{sec:airfoil_results}

We consider $N=500$ LF and HF samples. The average relative error in the mean and variance with a PC expansion of order $p=5$ is plotted in Figure \ref{fig:Air_BR}. At $r=3$ the BF approach provides a limited advantage over the HF and LF estimates. Accuracy of the BF solution saturates past $r=8$, implying that additional basis functions are not contributing novel information to the BF approximation. The BF estimate, however, performs better than the HF estimate for low $n$, particularly in the estimation of variance. The full PC expansion requires $P = 462$ terms.
\begin{figure*}[ht!]
\centering
\includegraphics[width=\textwidth]{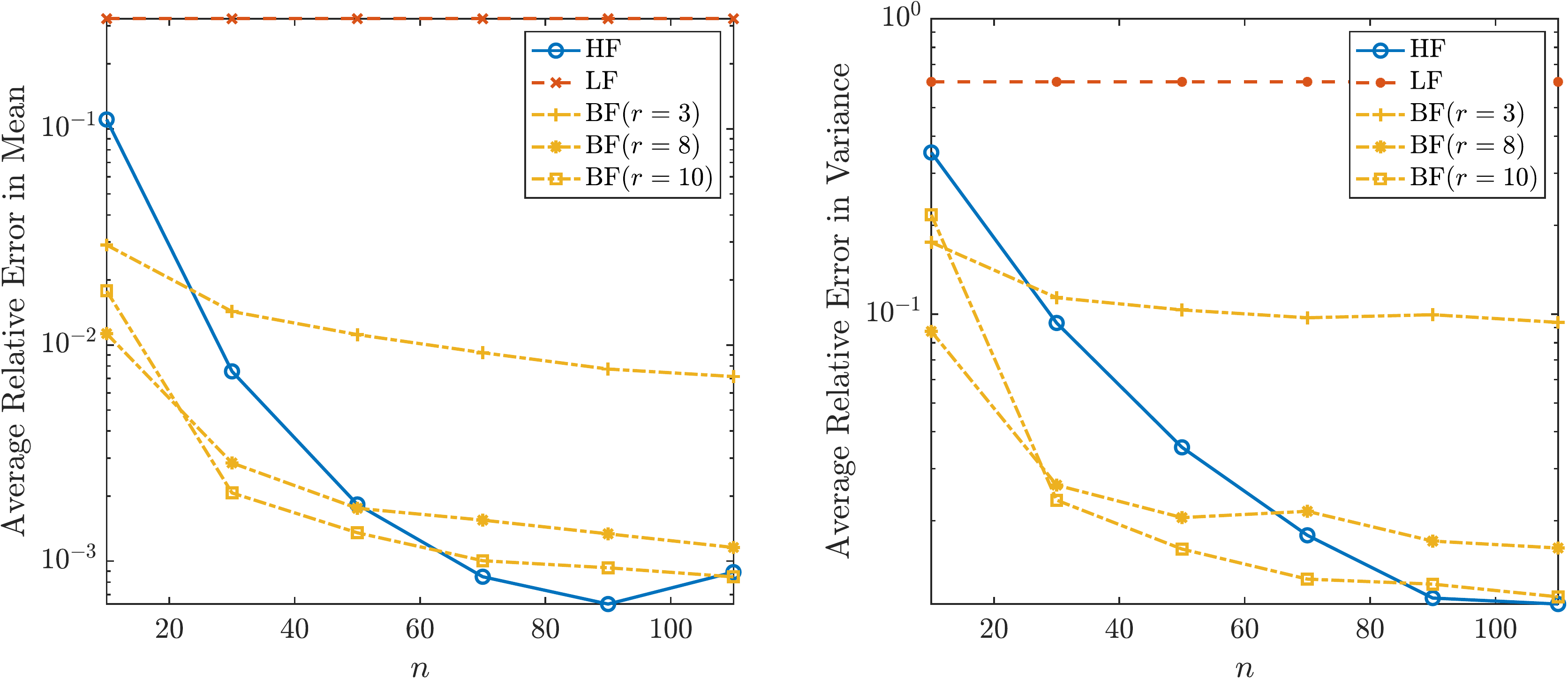}
\caption[The average relative error in (left) the mean, and (right) variance of $C_p$ estimation on the airfoil surface for Example III.]{The average relative error in (left) the mean, and (right) variance of $C_p$ estimation on the airfoil surface for Example III. Plotted are the HF, LF, and BF estimates for approximation rank $r$.}
\label{fig:Air_BR}	
\end{figure*}

Using an approximation rank of $r =8$ with a set of $n=20$ HF samples, we compare estimates of the mean and variance of the pressure coefficient as a function of normalized airfoil location in Figure \ref{fig:airfoil_mean_var}. Following \cite{skinner2019reduced}, the ``location on airfoil" is calculated as the stream-wise distance from the airfoil's trailing edge. Clockwise movement about the airfoil is defined as the positive direction, so positive ``location on airfoil" values correspond to the pressure surface and negative values to the suction surface. 
\begin{figure*}[ht!]
\centering
\includegraphics[width=\textwidth]{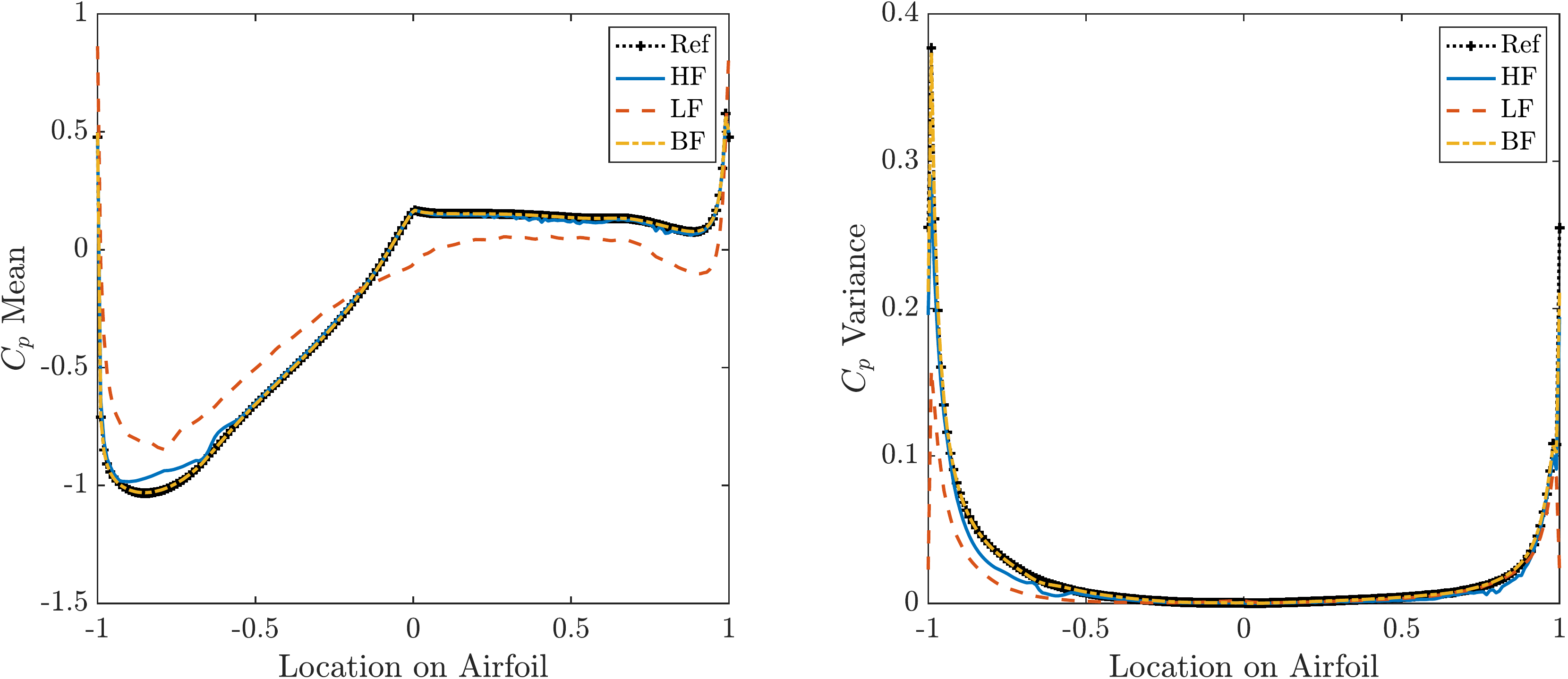}
\caption[Pressure coefficient on the NACA airfoil surface in Example III.]{Pressure coefficient on the NACA airfoil surface in Example III. Shown are the (left) mean and (right) variance for rank $r = 8$ with $n=20$ HF samples. The reference solution is given by Ref alongside the HF, LF and BF estimates.}
\label{fig:airfoil_mean_var}	
\end{figure*}
In Figure \ref{fig:airfoil_mean_var} the improvement from LF to BF is obvious, but the distinction between the HF and BF estimates is not possible to discern. Consulting Figure \ref{fig:Air_BR}, we find that although the BF estimate of the mean is more accurate than the HF, both have small relative errors less than $10^{-2}$ for $n=20$ and $r=8$. The distinction in variance estimates is closer to the order $10^{-1}$ but still difficult to discern in  Figure \ref{fig:airfoil_mean_var}, with the only noticeable deviation occurring in the airfoil location $ -1$ to $-0.5$. 

The eigenvalues of the LF and reference QoI covariance matrices are plotted in Figure \ref{fig:Air_Eig}. The eigenvalues of the LF and reference solutions align well and decay rapidly. The rapid decay indicates this problem is a suitable candidate for SMR, as corroborated in Figures \ref{fig:Air_BR} and \ref{fig:airfoil_mean_var}. Further, for indices greater than seven, there is some disagreement between the reference and LF eigenvalues, indicating that the corresponding basis functions may have a reduced influence on the accuracy of SMR. 
  \begin{figure}[ht!]
\centering
\includegraphics[width=0.5\textwidth]{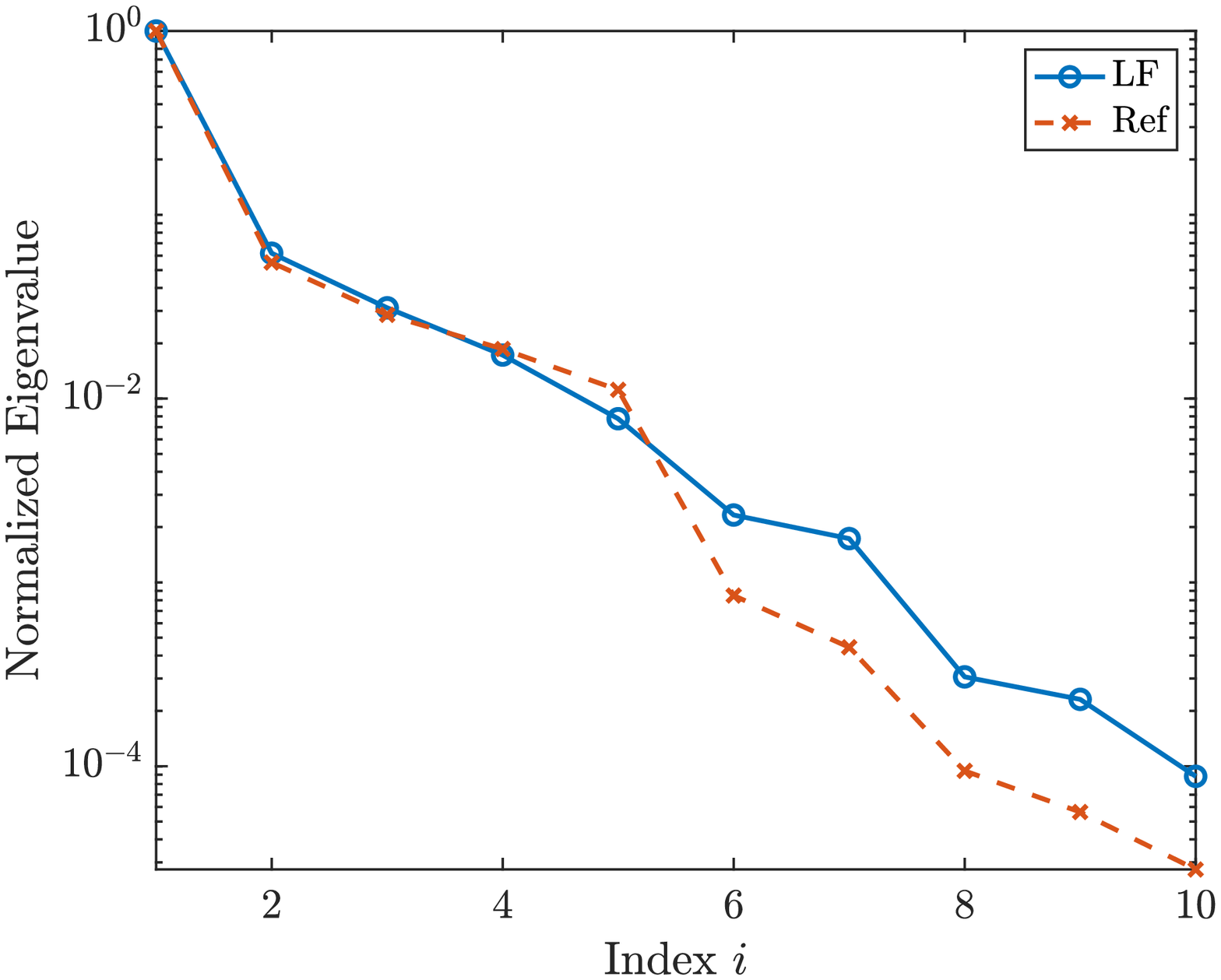}
\caption{Airfoil normalized eigenvalues of the LF and reference, Ref covariance matrices of airfoil surface $C_p$ for Example III.}
\label{fig:Air_Eig}	
\end{figure}

We next look at the performance of the practical error bound. We set $r=8$, $n = 20$ and $t=2.0$. 
In Table \ref{tab:airfoil_bound1} we present the efficacy of the practical error bound (\ref{eq:nu_bound_pract})  and associated probability (\ref{eq:prob_sum}). 
Similar to the preceding examples, we find that error is tightly bounded with high probability. We also note that the probabilities (\ref{eq:prob_sum}) and (\ref{eq:prob_vector}) are calculated to be equal. 
  \begin{table}[htbp]
  \centering
      \caption{Practical error probability (\ref{eq:prob_sum}) and bound efficacy (\ref{eq:nu_bound_pract}) in Example II. Results are calculated as the average of $30$ repetitions.}
    \begin{tabular}{ccccccc}
    \hline\noalign{\smallskip}
  QoI   & $r$ & $n$ & $\hat{n}$ & $N$ & Prob. (\ref{eq:prob_sum}) & Eff. (\ref{eq:nu_bound_pract})  \\
    \noalign{\smallskip}\hline\noalign{\smallskip}
 $C_p$  & $8$ & $20$ & $20$ & $500$ & $0.891$  & $1.40$\\
    \noalign{\smallskip}\hline
    \end{tabular}%

  \label{tab:airfoil_bound1}%
\end{table}
Figure \ref{fig:airfoil} (a)  presents the pointwise reference error and error bound (\ref{eq:nu_bound_i_pract}). The bound captures the shape of the pointwise error very well, remaining tight across all points of the airfoil. Figure \ref{fig:airfoil} (b) illustrates the efficacy for a given $(n,r)$ pair as the average of $30$ repetitions. Consistent with the lid-driven cavity and gas turbine efficacy plots, we find there is a correlation between the number of HF samples $n$ and the efficacy of the error bound. Interestingly, a relationship is also evident between the rank $r$ and the error bound efficacy, namely errors at lower ranks may be accurately estimated via few HF samples. For all $(n,r)$ pairs the efficacy of the practical error bound for both QoI does not exceed more than two implying that error bound (\ref{eq:nu_bound_pract}) is a useful tool to calibrate the BF estimate. 
\begin{figure*}[ht!]
\centering
\begin{subfigure}{0.45\textwidth}
  \centering
  \includegraphics[width=\textwidth]{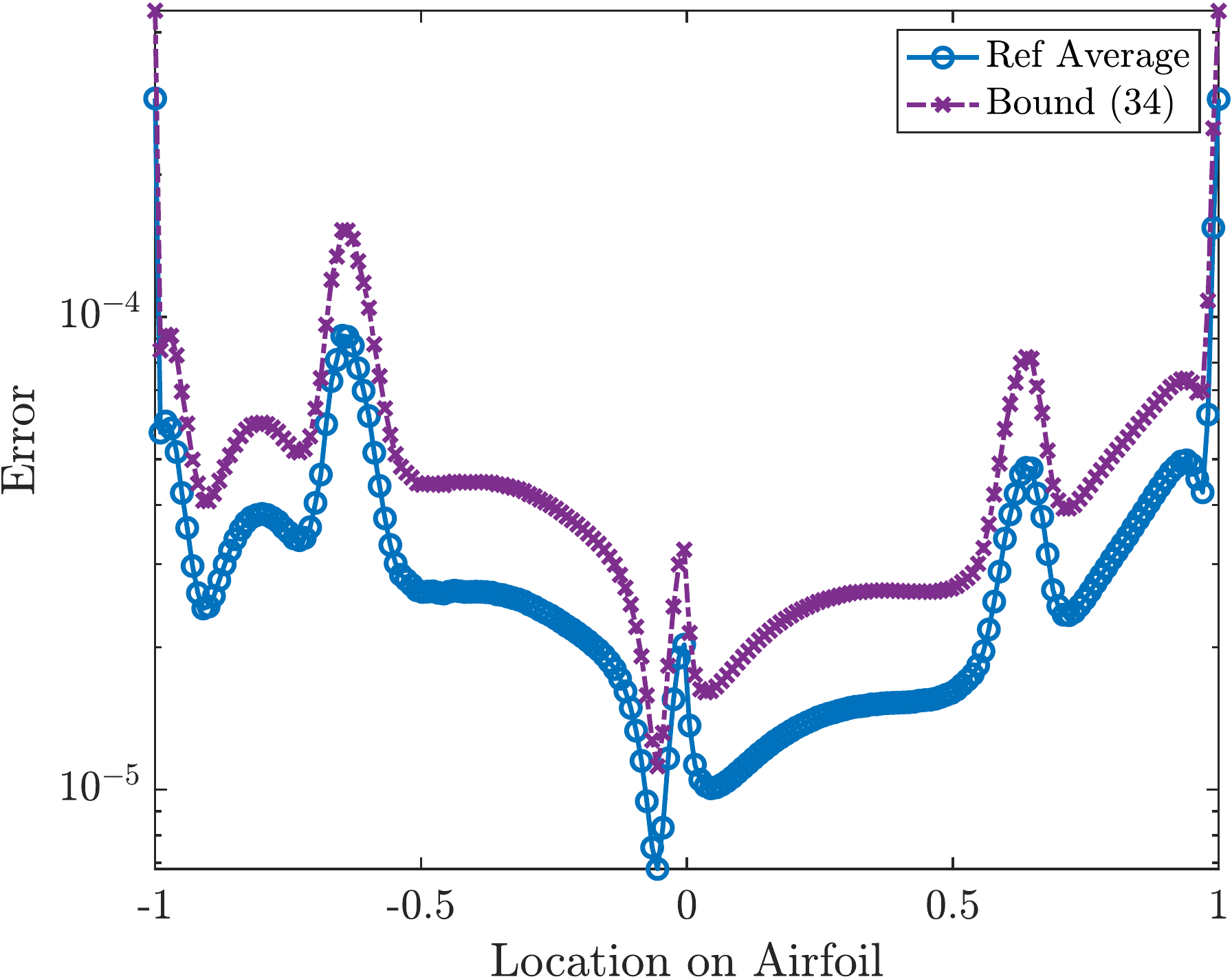}
  \caption{Pointwise true error and error bound (\ref{eq:nu_bound_i_pract}).  }
  \label{fig:airfoil_bound}
\end{subfigure}
\hfill
\begin{subfigure}{0.45\textwidth}
  \centering
  \includegraphics[width=\textwidth]{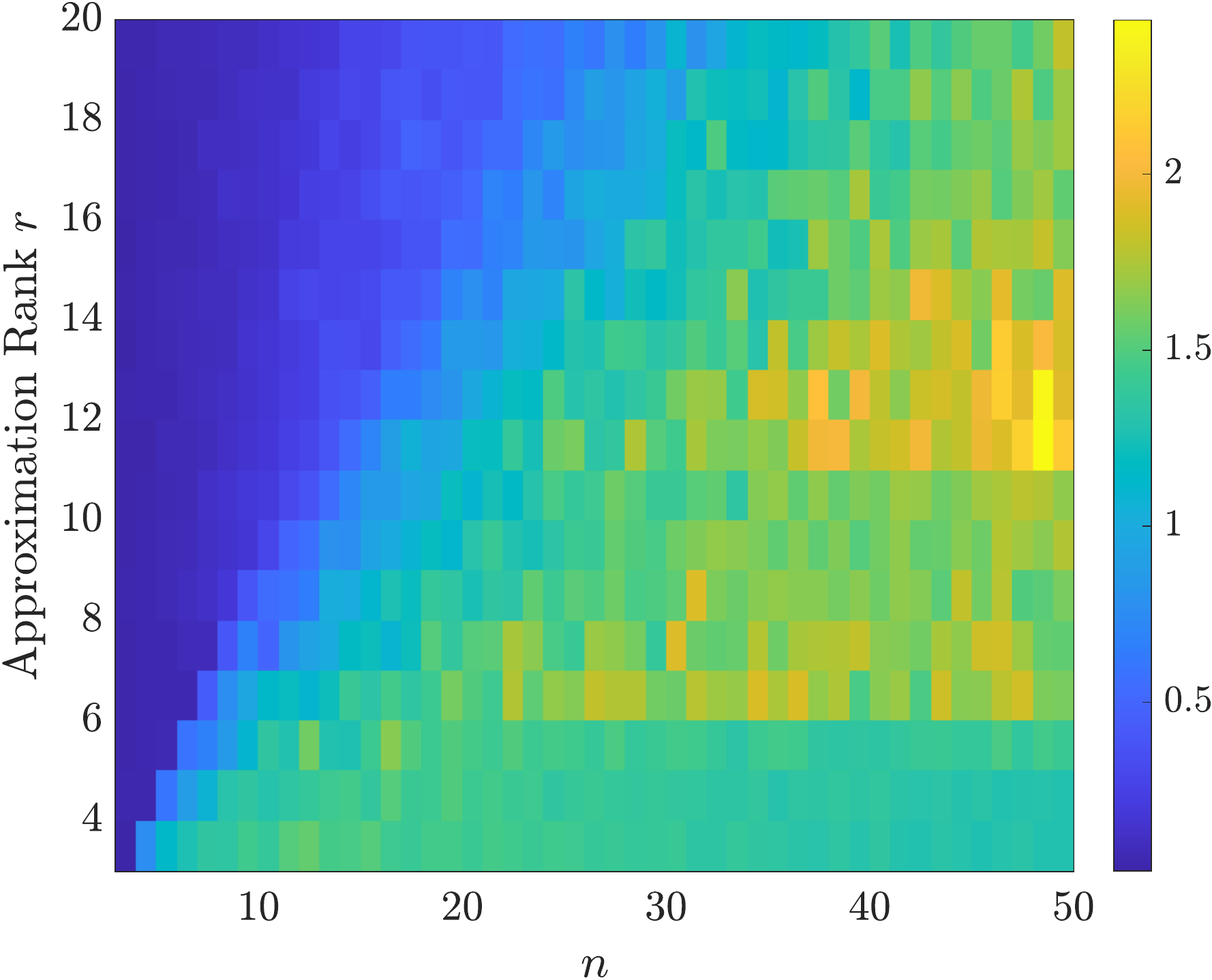}
  \caption{Practical error bound efficacy.}
  \label{fig:airfoil_efficacy}
\end{subfigure}
 \caption[Error bound performance for the airfoil of Example III.]{Error bound performance for the airfoil of Example III. We present the average calculated from $30$ repetitions, where each repetition includes the estimate of $N=500$ samples. In (a) we use $n = 20$ HF samples and rank $r= 8$. In (b) the practical bound efficacy is calculated as the (average) ratio of the error estimated from the practical error bound (\ref{eq:nu_bound_pract}) and the true error for different approximation ranks $r$ and number of HF samples $n$. }
 \label{fig:airfoil}
\end{figure*}
\section{Conclusions}
\label{sec:conclusion_pc}

{\color{black}In this work, we present a BF stochastic model reduction (SMR) approach to approximate the solutions to PDEs with parametric or stochastic inputs. The BF estimate is obtained through forming a PC expansion of the LF solution of interest, truncating a KL expansion of the LF PC expansion to form a reduced basis before regressing a limited number of expensive HF samples against this reduced basis. Once built, this regression model can be employed to generate solution statistics.

We develop two types of novel error bounds for the BF SMR approximation. The first error estimate identifies a requirement on a pair of LF and HF models towards a successful BF approximation. We also present a second error bound from  a practical perspective to estimate the SMR error using a limited number of HF samples.  To the best of our knowledge these are the first error bounds that address this SMR algorithm.

We demonstrate the effectiveness of the SMR approach along with the efficacy of the practical error estimate in three numerical examples. In the first example, application of the BF approach to the lid-driven cavity achieves an order of magnitude better accuracy while maintaining the same computational expense. In the second example of heated flow past a cylinder we observe similar improvement in the estimation of the temperature on the cylinder surface, and limited improvement in temperature along a vertical line at $x=0.2$, demonstrating the choice of QoI to be an important factor in the BF SMR approach. In the third example modeling geometric variability of an airfoil, we again observe the successful application of the BF approach. In addition, we find in all three examples that evaluation of the practical error bound using a limited number of HF samples provides a useful and accurate estimation of the BF SMR error. }

Our future research includes employing optimal sampling strategies to enable more intelligent selection of HF samples when generating the BF estimates and the extension of SMR to account for multiple levels of model fidelity. Additionally, LF model design, where a LF model is calibrated to achieve the best possible BF estimate, is an interesting line of inquiry. 

\section*{Acknowledgments} 
The work of FN was supported by NSF grants 1740330 and 2028032. JH has received funding from the European Union's Horizon 2020 research and innovation programme under the Marie Sk\l{}odowska-Curie grant agreement No 712949 (TECNIOspring PLUS) and from the Agency for Business Competitiveness of the Government of Catalonia. The work of AD was also supported by the AFOSR grant FA9550-20-1-0138 and NSF grant 1454601. 

\section{Declarations}
\label{sec:declarations}
%

%
%
\subsection{Funding}

The work of FN was supported by NSF grants 1740330 and 2028032. JH has received funding from the European Union's Horizon 2020 research and innovation programme under the Marie Sk\l{}odowska-Curie grant agreement No 712949 (TECNIOspring PLUS) and from the Agency for Business Competitiveness of the Government of Catalonia. The work of AD was also supported by the AFOSR grant FA9550-20-1-0138.

\subsection{Conflicts of interest/Competing interests}

The authors have no relevant financial or non-financial interests to disclose.

\subsection{Availability of data and material}
%

The datasets generated during and/or analyzed during the current study will be uploaded to the GitHub page \url{www.github.com/CU-UQ} once the paper is published. 

%
%
\subsection{Code availability}

The  code used in this study will be made available once the paper is published at the GitHub page \url{www.github.com/CU-UQ}. 

%
%
\subsection{Ethics approval}

Not applicable. 

\subsection{Consent to participate}

Not applicable. 

\subsection{Consent for publication}

Not applicable. 

\bibliographystyle{spbasic}
\bibliography{mybibfile}

\begin{thebibliography}{61}
\providecommand{\natexlab}[1]{#1}
\providecommand{\url}[1]{{#1}}
\providecommand{\urlprefix}{URL }
\expandafter\ifx\csname urlstyle\endcsname\relax
  \providecommand{\doi}[1]{DOI~\discretionary{}{}{}#1}\else
  \providecommand{\doi}{DOI~\discretionary{}{}{}\begingroup
  \urlstyle{rm}\Url}\fi
\providecommand{\eprint}[2][]{\url{#2}}

\bibitem[{Aln{\ae}s et~al.(2015)Aln{\ae}s, Blechta, Hake, Johansson, Kehlet,
  Logg, Richardson, Ring, Rognes, and Wells}]{AlnaesBlechta2015a}
Aln{\ae}s MS, Blechta J, Hake J, Johansson A, Kehlet B, Logg A, Richardson C,
  Ring J, Rognes ME, Wells GN (2015) The fenics project version 1.5. Archive of
  Numerical Software 3(100)

\bibitem[{Berry(1941)}]{berry1941accuracy}
Berry AC (1941) The accuracy of the gaussian approximation to the sum of
  independent variates. Transactions of the american mathematical society
  49(1):122--136

\bibitem[{Blatman and Sudret(2010)}]{blatman2010adaptive}
Blatman G, Sudret B (2010) An adaptive algorithm to build up sparse polynomial
  chaos expansions for stochastic finite element analysis. Probabilistic
  Engineering Mechanics 25(2):183--197

\bibitem[{Blatman and Sudret(2011)}]{blatman2011adaptive}
Blatman G, Sudret B (2011) Adaptive sparse polynomial chaos expansion based on
  least angle regression. Journal of Computational Physics 230(6):2345--2367

\bibitem[{Cand{\`e}s and Wakin(2008)}]{candes2008introduction}
Cand{\`e}s EJ, Wakin MB (2008) An introduction to compressive sampling. IEEE
  signal processing magazine 25(2):21--30

\bibitem[{Chandrasheka(2020)}]{chandGithub}
Chandrasheka P (2020) \urlprefix\url{http://cpraveen.github.io}

\bibitem[{Cheng et~al.(2005)Cheng, Gimbutas, Martinsson, and
  Rokhlin}]{cheng2005compression}
Cheng H, Gimbutas Z, Martinsson PG, Rokhlin V (2005) On the compression of low
  rank matrices. SIAM Journal on Scientific Computing 26(4):1389--1404

\bibitem[{Cliffe et~al.(2011)Cliffe, Giles, Scheichl, and
  Teckentrup}]{cliffe2011multilevel}
Cliffe KA, Giles MB, Scheichl R, Teckentrup AL (2011) Multilevel monte carlo
  methods and applications to elliptic pdes with random coefficients. Computing
  and Visualization in Science 14(1):3

\bibitem[{Cohen et~al.(2013)Cohen, Davenport, and
  Leviatan}]{cohen2013stability}
Cohen A, Davenport MA, Leviatan D (2013) On the stability and accuracy of least
  squares approximations. Foundations of Ccomputational Mathematics
  13(5):819--834

\bibitem[{Constantine et~al.(2009)Constantine, Doostan, and
  Iaccarino}]{constantine2009hybrid}
Constantine PG, Doostan A, Iaccarino G (2009) A hybrid collocation/galerkin
  scheme for convective heat transfer problems with stochastic boundary
  conditions. International Journal for Numerical Methods in Engineering
  80(6-7):868--880

\bibitem[{De et~al.(2020)De, Britton, Reynolds, Skinner, Jansen, and
  Doostan}]{de2020transfer}
De S, Britton J, Reynolds M, Skinner R, Jansen K, Doostan A (2020) On transfer
  learning of neural networks using bi-fidelity data for uncertainty
  propagation. International Journal for Uncertainty Quantification 10(6)

\bibitem[{Diaz et~al.(2018)Diaz, Doostan, and Hampton}]{diaz2018sparse}
Diaz P, Doostan A, Hampton J (2018) Sparse polynomial chaos expansions via
  compressed sensing and d-optimal design. Computer Methods in Applied
  Mechanics and Engineering 336:640--666

\bibitem[{Diskin et~al.(2015)Diskin, Thomas, Rumsey, and
  Schw{\"o}ppe}]{diskin2015grid}
Diskin B, Thomas J, Rumsey CL, Schw{\"o}ppe A (2015) Grid convergence for
  turbulent flows. In: 53rd AIAA Aerospace Sciences Meeting, p 1746

\bibitem[{Donoho(2006)}]{donoho2006compressed}
Donoho DL (2006) Compressed sensing. IEEE Transactions on Information Theory
  52(4):1289--1306

\bibitem[{Doostan and Owhadi(2011)}]{doostan2011non}
Doostan A, Owhadi H (2011) A non-adapted sparse approximation of pdes with
  stochastic inputs. Journal of Computational Physics 230(8):3015--3034

\bibitem[{Doostan et~al.(2007)Doostan, Ghanem, and
  Red-Horse}]{doostan2007stochastic}
Doostan A, Ghanem RG, Red-Horse J (2007) Stochastic model reduction for chaos
  representations. Computer Methods in Applied Mechanics and Engineering
  196(37-40):3951--3966

\bibitem[{Doostan et~al.(2016)Doostan, Geraci, and Iaccarino}]{doostan2016bi}
Doostan A, Geraci G, Iaccarino G (2016) A bi-fidelity approach for uncertainty
  quantification of heat transfer in a rectangular ribbed channel. In: ASME
  Turbo Expo 2016: Turbomachinery Technical Conference and Exposition, American
  Society of Mechanical Engineers, pp V02CT45A031--V02CT45A031

\bibitem[{Eldred(2009)}]{eldred2009recent}
Eldred M (2009) Recent advances in non-intrusive polynomial chaos and
  stochastic collocation methods for uncertainty analysis and design. In: 50th
  AIAA/ASME/ASCE/AHS/ASC Structures, Structural Dynamics, and Materials
  Conference 17th AIAA/ASME/AHS Adaptive Structures Conference 11th AIAA No, p
  2274

\bibitem[{Fairbanks et~al.(2017)Fairbanks, Doostan, Ketelsen, and
  Iaccarino}]{fairbanks2017low}
Fairbanks HR, Doostan A, Ketelsen C, Iaccarino G (2017) A low-rank control
  variate for multilevel monte carlo simulation of high-dimensional uncertain
  systems. Journal of Computational Physics 341:121--139

\bibitem[{Fairbanks et~al.(2020)Fairbanks, Jofre, Geraci, Iaccarino, and
  Doostan}]{fairbanks2020bi}
Fairbanks HR, Jofre L, Geraci G, Iaccarino G, Doostan A (2020) Bi-fidelity
  approximation for uncertainty quantification and sensitivity analysis of
  irradiated particle-laden turbulence. Journal of Computational Physics
  402:108996

\bibitem[{Fern{\'a}ndez-Godino et~al.(2016)Fern{\'a}ndez-Godino, Park, Kim, and
  Haftka}]{fernandez2016review}
Fern{\'a}ndez-Godino MG, Park C, Kim NH, Haftka RT (2016) Review of
  multi-fidelity models. arXiv preprint arXiv:160907196

\bibitem[{Forrester et~al.(2007)Forrester, S{\'o}bester, and
  Keane}]{forrester2007multi}
Forrester AI, S{\'o}bester A, Keane AJ (2007) Multi-fidelity optimization via
  surrogate modelling. Proceedings of the Royal Society A: Mathematical,
  Physical and Engineering Sciences 463(2088):3251--3269

\bibitem[{Ghanem et~al.(2007)Ghanem, Saad, and Doostan}]{ghanem2007efficient}
Ghanem R, Saad G, Doostan A (2007) Efficient solution of stochastic systems:
  application to the embankment dam problem. Structural Safety 29(3):238--251

\bibitem[{Ghanem and Spanos(1991)}]{ghanem1991stochastic}
Ghanem RG, Spanos PD (1991) Stochastic finite element method: Response
  statistics. In: Stochastic Finite Elements: A Spectral Approach, Springer, pp
  101--119

\bibitem[{Ghia et~al.(1982)Ghia, Ghia, and Shin}]{ghia1982high}
Ghia U, Ghia KN, Shin C (1982) High-re solutions for incompressible flow using
  the navier-stokes equations and a multigrid method. Journal of Computational
  Physics 48(3):387--411

\bibitem[{Giles(2008)}]{giles2008multilevel}
Giles MB (2008) Multilevel monte carlo path simulation. Operations Research
  56(3):607--617

\bibitem[{Giles(2013)}]{giles2013multilevel}
Giles MB (2013) Multilevel monte carlo methods. In: Monte Carlo and Quasi-Monte
  Carlo Methods 2012, Springer, pp 83--103

\bibitem[{Gu and Eisenstat(1996)}]{gu1996efficient}
Gu M, Eisenstat SC (1996) Efficient algorithms for computing a strong
  rank-revealing qr factorization. SIAM Journal on Scientific Computing
  17(4):848--869

\bibitem[{Hampton and Doostan(2015{\natexlab{a}})}]{hampton2015coherence}
Hampton J, Doostan A (2015{\natexlab{a}}) Coherence motivated sampling and
  convergence analysis of least squares polynomial chaos regression. Computer
  Methods in Applied Mechanics and Engineering 290:73--97

\bibitem[{Hampton and Doostan(2015{\natexlab{b}})}]{hampton2015compressive}
Hampton J, Doostan A (2015{\natexlab{b}}) Compressive sampling of polynomial
  chaos expansions: Convergence analysis and sampling strategies. Journal of
  Computational Physics 280:363--386

\bibitem[{Hampton et~al.(2018)Hampton, Fairbanks, Narayan, and
  Doostan}]{hampton2018practical}
Hampton J, Fairbanks HR, Narayan A, Doostan A (2018) Practical error bounds for
  a non-intrusive bi-fidelity approach to parametric/stochastic model
  reduction. Journal of Computational Physics 368:315--332

\bibitem[{Hill and Peterson(1992)}]{hill1992mechanics}
Hill PG, Peterson CR (1992) Mechanics and thermodynamics of propulsion.
  Reading, MA, Addison-Wesley Publishing Co, 1992, 764 p

\bibitem[{Kennedy and O'Hagan(2000)}]{kennedy2000predicting}
Kennedy MC, O'Hagan A (2000) Predicting the output from a complex computer code
  when fast approximations are available. Biometrika 87(1):1--13

\bibitem[{Kleiber et~al.(2013)Kleiber, Sain, Heaton, Wiltberger, Reese, Bingham
  et~al.}]{kleiber2013parameter}
Kleiber W, Sain SR, Heaton MJ, Wiltberger M, Reese CS, Bingham D, et~al. (2013)
  Parameter tuning for a multi-fidelity dynamical model of the magnetosphere.
  The Annals of Applied Statistics 7(3):1286--1310

\bibitem[{Laurenceau and Sagaut(2008)}]{laurenceau2008building}
Laurenceau J, Sagaut P (2008) Building efficient response surfaces of
  aerodynamic functions with kriging and cokriging. AIAA Journal 46(2):498--507

\bibitem[{Le~Gratiet and Cannamela(2015)}]{le2015cokriging}
Le~Gratiet L, Cannamela C (2015) Cokriging-based sequential design strategies
  using fast cross-validation techniques for multi-fidelity computer codes.
  Technometrics 57(3):418--427

\bibitem[{Le~Gratiet and Garnier(2014)}]{le2014recursive}
Le~Gratiet L, Garnier J (2014) Recursive co-kriging model for design of
  computer experiments with multiple levels of fidelity. International Journal
  for Uncertainty Quantification 4(5)

\bibitem[{Le~Ma{\^\i}tre and Knio(2010)}]{le2010spectral}
Le~Ma{\^\i}tre O, Knio OM (2010) Spectral methods for uncertainty
  quantification: with applications to computational fluid dynamics. Springer
  Science \& Business Media

\bibitem[{Lyall et~al.(2011)Lyall, Thrift, Thole, and Kohli}]{lyall2011heat}
Lyall ME, Thrift AA, Thole KA, Kohli A (2011) Heat transfer from low aspect
  ratio pin fins. Journal of Turbomachinery 133(1):011001

\bibitem[{Martinsson and Tygert(2011)}]{martinsson2011randomized}
Martinsson R, Tygert (2011) A randomized algorithm for the decomposition of
  matrices. Applied and Computational Harmonic Analysis 30(1):47--68

\bibitem[{Mathelin and Gallivan(2012)}]{mathelin2012compressed}
Mathelin L, Gallivan K (2012) A compressed sensing approach for partial
  differential equations with random input data. Communications in
  Computational Physics 12(4):919--954

\bibitem[{Narayan et~al.(2014)Narayan, Gittelson, and
  Xiu}]{narayan2014stochastic}
Narayan A, Gittelson C, Xiu D (2014) A stochastic collocation algorithm with
  multifidelity models. SIAM Journal on Scientific Computing 36(2):A495--A521

\bibitem[{Ng and Eldred(2012)}]{ng2012multifidelity}
Ng LWT, Eldred M (2012) Multifidelity uncertainty quantification using
  non-intrusive polynomial chaos and stochastic collocation. In: 53rd
  AIAA/ASME/ASCE/AHS/ASC Structures, Structural Dynamics and Materials
  Conference 20th AIAA/ASME/AHS Adaptive Structures Conference 14th AIAA, p
  1852

\bibitem[{Padron et~al.(2016)Padron, Alonso, and Eldred}]{padron2016multi}
Padron AS, Alonso JJ, Eldred MS (2016) Multi-fidelity methods in aerodynamic
  robust optimization. In: 18th AIAA Non-Deterministic Approaches Conference, p
  0680

\bibitem[{Palar et~al.(2015)Palar, Tsuchiya, and
  Parks}]{palar2015decomposition}
Palar PS, Tsuchiya T, Parks G (2015) Decomposition-based evolutionary
  aerodynamic robust optimization with multi-fidelity point collocation
  non-intrusive polynomial chaos. In: 17th AIAA Non-Deterministic Approaches
  Conference, p 1377

\bibitem[{Parussini et~al.(2017)Parussini, Venturi, Perdikaris, and
  Karniadakis}]{parussini2017multi}
Parussini L, Venturi D, Perdikaris P, Karniadakis GE (2017) Multi-fidelity
  gaussian process regression for prediction of random fields. Journal of
  Computational Physics 336:36--50

\bibitem[{Peherstorfer et~al.(2018)Peherstorfer, Willcox, and
  Gunzburger}]{peherstorfer2018survey}
Peherstorfer B, Willcox K, Gunzburger M (2018) Survey of multifidelity methods
  in uncertainty propagation, inference, and optimization. Siam Review
  60(3):550--591

\bibitem[{Peng et~al.(2014)Peng, Hampton, and Doostan}]{peng2014weighted}
Peng J, Hampton J, Doostan A (2014) A weighted $\ell$1-minimization approach
  for sparse polynomial chaos expansions. Journal of Computational Physics
  267:92--111

\bibitem[{Perdikaris et~al.(2015)Perdikaris, Venturi, Royset, and
  Karniadakis}]{perdikaris2015multi}
Perdikaris P, Venturi D, Royset JO, Karniadakis GE (2015) Multi-fidelity
  modelling via recursive co-kriging and gaussian--markov random fields.
  Proceedings of the Royal Society A: Mathematical, Physical and Engineering
  Sciences 471(2179):20150018

\bibitem[{Perdikaris et~al.(2016)Perdikaris, Venturi, and
  Karniadakis}]{perdikaris2016multifidelity}
Perdikaris P, Venturi D, Karniadakis GE (2016) Multifidelity information fusion
  algorithms for high-dimensional systems and massive data sets. SIAM Journal
  on Scientific Computing 38(4):B521--B538

\bibitem[{Pope(2001)}]{pope2001turbulent}
Pope SB (2001) Turbulent flows

\bibitem[{Raisee et~al.(2015)Raisee, Kumar, and Lacor}]{raisee2015non}
Raisee M, Kumar D, Lacor C (2015) A non-intrusive model reduction approach for
  polynomial chaos expansion using proper orthogonal decomposition.
  International Journal for Numerical Methods in Engineering 103(4):293--312

\bibitem[{Shevtsova(2011)}]{shevtsova2011absolute}
Shevtsova I (2011) On the absolute constants in the berry-esseen type
  inequalities for identically distributed summands. arXiv preprint
  arXiv:11116554

\bibitem[{Skinner et~al.(2019)Skinner, Doostan, Peters, Evans, and
  Jansen}]{skinner2019reduced}
Skinner RW, Doostan A, Peters EL, Evans JA, Jansen KE (2019) Reduced-basis
  multifidelity approach for efficient parametric study of naca airfoils. AIAA
  Journal 57(4):1481--1491

\bibitem[{Spalart and Allmaras(1992)}]{spalart1992one}
Spalart P, Allmaras S (1992) A one-equation turbulence model for aerodynamic
  flows. In: 30th aerospace sciences meeting and exhibit, p 439

\bibitem[{Whiting and Jansen(2001)}]{whiting2001stabilized}
Whiting CH, Jansen KE (2001) A stabilized finite element method for the
  incompressible navier--stokes equations using a hierarchical basis.
  International Journal for Numerical Methods in Fluids 35(1):93--116

\bibitem[{Xiu(2010)}]{xiu2010numerical}
Xiu D (2010) Numerical methods for stochastic computations: a spectral method
  approach. Princeton university press

\bibitem[{Xiu and Karniadakis(2002)}]{xiu2002wiener}
Xiu D, Karniadakis GE (2002) The wiener--askey polynomial chaos for stochastic
  differential equations. SIAM Journal on Scientific Computing 24(2):619--644

\bibitem[{Yan et~al.(2012)Yan, Guo, and Xiu}]{yan2012stochastic}
Yan L, Guo L, Xiu D (2012) Stochastic collocation algorithms using
  $\ell$1-minimization. International Journal for Uncertainty Quantification
  2(3)

\bibitem[{Yang and Karniadakis(2013)}]{yang2013reweighted}
Yang X, Karniadakis GE (2013) Reweighted $\ell$1 minimization method for
  stochastic elliptic differential equations. Journal of Computational Physics
  248:87--108

\bibitem[{Zhu et~al.(2014)Zhu, Narayan, and Xiu}]{zhu2014computational}
Zhu X, Narayan A, Xiu D (2014) Computational aspects of stochastic collocation
  with multifidelity models. SIAM/ASA Journal on Uncertainty Quantification
  2(1):444--463

\end{thebibliography}

\end{document}